\documentclass{winnower}
\usepackage{amsmath,amsthm,amssymb,graphicx}
\usepackage[lined,boxed,commentsnumbered]{algorithm2e}
\usepackage{pgfplots}
\usepackage{comment}
\usepackage{bm}
\usepackage{subcaption}

\renewcommand{\vec}[1]{\mathbf{#1}}

\newcommand{\vu}{\bm{u}}
\newcommand{\va}{\bm{a}}
\newcommand{\ds}{\displaystyle}
\newcommand{\vs}{\vspace{0.5ex}}

\theoremstyle{definition}

\theoremstyle{remark}

\newtheorem{theorem}{Theorem}

\newtheorem{lemma}[theorem]{Lemma}

\begin{document}
	
\title{Balancing Fairness and Efficiency in an Optimization Model}
	
\author{Violet (Xinying) Chen and J. N. Hooker \\ Carnegie Mellon University \\ \texttt{xinyingc@andrew.cmu.edu, jh38@andrew.cmu.edu}}
	
\date{June 2020}
	
\maketitle
	
\begin{abstract}
Optimization models generally aim for efficiency by maximizing total benefit or minimizing cost.  Yet a trade-off between fairness and efficiency is an important element of many practical decisions.  We propose a principled and practical method for balancing these two criteria in an optimization model. Following a critical assessment of existing schemes, we define a set of social welfare functions (SWFs) that combine Rawlsian leximax fairness and utilitarianism and overcome some of the weaknesses of previous approaches.  In particular, we regulate the equity/efficiency trade-off with a single parameter that has a meaningful interpretation in practical contexts.  We formulate the SWFs using mixed integer constraints and sequentially maximize them subject to constraints that define the problem at hand.  After providing practical step-by-step instructions for implementation, we demonstrate the method on problems of realistic size involving healthcare resource allocation and disaster preparation.  The solution times are modest, ranging from a fraction of a second to 18 seconds for a given value of the trade-off parameter. 
\end{abstract}

\section{Introduction}
	


Fairness is an important consideration across a wide range of optimization models.  It can be a central issue in health care provision, disaster planning, workload allocation, public facility location, telecommunication network management, traffic signal timing, and many other contexts.  While it is normally straightforward to formulate an objective function that reflects efficiency or cost, it is not obvious how to express fairness in mathematical form.  When both fairness and efficiency are desired, as is typical in practice, there is the additional challenge of mathematically integrating them in a tractable model.  

We undertake in this paper to develop a practical and yet principled approach to balancing fairness and efficiency that can be implemented in a mixed \mbox{integer}/linear programming (MILP) model.  We present the underlying theory as well as practical guidelines for implementation.  In particular, we provide step-by-step instructions for incorporating our formulation into an existing model (Section~\ref{sec:guide}). 

Our guiding intuition is that a scheme that balances equity and efficiency should give some degree of priority to parties in greatest need, but not at unlimited cost.  When a natural disaster brings down the electric power grid, crisis managers may dispatch crews to urban areas first in order to restore power to more households quickly, thus maximizing efficiency. Yet this may cause rural areas to experience very long blackouts, which could be seen as unfair.  
A more satisfactory solution might give some amount of priority to rural customers, but without imposing too much harm on the population as a whole.  Similarly, traffic signal timing that minimizes total delay may result in impracticably long wait times for traffic on minor streets that cross a main thoroughfare.  Again a balance between equity and efficiency may be desirable.  The issue can be especially acute in health care.  Expensive treatments or research programs that prolong the life of a relatively few gravely ill patients may divert funds from preventive health measures that would spare thousands the suffering brought by less serious diseases.  The obvious question in such cases is how to regulate the trade-off between fairness and efficiency.  We propose an approach that is arguably more easily interpreted and implemented in a practical optimization model than other schemes that have been proposed.


While there are many possible measures of fairness, we choose a criterion based ultimately on John Rawls' concept of justice-as-fairness (\citeauthor{Raw99} \citeyear{Raw99}).  One consequence of the Rawlsian analysis is his famous difference principle, which states roughly that a fair distribution of resources is one that maximizes the welfare of the worst-off.  Rawls defends the principle with a social contract argument that can be plausibly extended to lexicographic maximization.  That is, the welfare of the worst-off is first maximized subject to resource constraints, whereupon the welfare of the second worst-off is maximized while holding that of the worst-off fixed, and so forth.  We choose a lexicographic formulation because it provides a more nuanced measure of fairness than the maximin criterion, and because the Rawlsian perspective
has been defended by closely reasoned philosophical arguments in a vast literature (\citeauthor{RicWei99} \citeyear{RicWei99}, \citeauthor{Fre03} \citeyear{Fre03}).  Our aim is to combine this fairness criterion with an efficiency measure in a principled and practical fashion.

The Rawlsian argument goes roughly as follows.  Let's suppose that all concerned parties adopt an agreed-upon social policy in an original position behind a ``veil of ignorance'' as to their identity.  It must be a policy that all parties can rationally accept upon learning who they are.  Rawls argues that no rational decision maker will accept a policy in which she is the least advantaged, unless she would have been even worse off under any other policy.  A fair outcome should therefore maximize the welfare of the worst-off.  The argument can be employed recursively to defend a leximax criterion.  Rawls intended his principle to apply only to the design of social institutions, and to pertain only to the distribution of ``primary goods,'' which are goods that any rational person would want.  Yet it can be plausibly extended to distributive justice in general, particularly if it is appropriately combined with an efficiency criterion.  

Fairness can be incorporated into an optimization model by using a {\em social welfare function} (SWF) as the objective function.  The SWF is a function of the utility levels allocated to affected parties, where utility can be defined as profit, negative cost, or any other measure of benefit that suits the context.  The SWF assigns higher values to more desirable utility distributions, and the optimization problem is to maximize the SWF subject to problem constraints.  The goal is to develop a SWF that reflects both fairness and efficiency. 

Our contribution is threefold.  We begin by assessing, from the perspective of an optimization modeler, several existing approaches to balancing fairness and efficiency in a SWF.  We then propose a scheme that not only has a tractable MILP formulation but, in our view, overcomes several weaknesses of existing methods.  In particular, we address a problem that nearly all of them share: a failure to parameterize the trade-off between efficiency and fairness in a fashion that can be understood and applied in practice.  Finally, we conduct a thorough analysis of the mathematical properties of the formulation proposed here.

One modeling scheme that offers a potentially appealing approach to the fairness/efficiency trade-off is that of \citeauthor{HooWil12} (\citeyear{HooWil12}).  Their SWF contains a single parameter $\Delta$ that has the same units as utility and can be related naturally to the problem at hand.  The value of $\Delta$ is chosen so that parties whose utility is within $\Delta$ of the lowest are seen as sufficiently disadvantaged to deserve priority.  This yields a utilitarian criterion when $\Delta=0$ and a pure maximin criterion for sufficiently large $\Delta$, with intermediate values of $\Delta$ injecting the fairness criterion to varying degrees.  


Despite its attractive features, the Hooker-Williams scheme has shortcomings.  One is that the actual utility levels of disadvantaged parties other than the very worst-off have no bearing on social welfare.  The SWF imposes no penalty for reducing their utility levels to that of the worst-off.  This leads to solutions that may not adequately account for equity considerations in real applications.  

We address this problem by combining efficiency with a lexicographic criterion rather than a maximin criterion.  This allows the utility levels of all disadvantaged parties to factor into social welfare.  It also results in a more sophisticated interpretation of the trade-off parameter $\Delta$.  However, it poses the difficult challenge of designing an SWF that yields a practical optimization model.  We meet the challenge by maximizing a series of SWFs rather than a single function.  We show how to incorporate each SWF as the objective function of the given problem using MILP constraints.  We then obtain a socially optimal solution by solving the resulting sequence of optimization models.  

The paper is organized as follows.  We begin in Section~\ref{sec:lit} with an assessment of the primary existing schemes for combining equity and efficiency in a SWF.  These include convex and other combinations of utility with such equity criteria as the Gini cofficient and the Rawlsian difference principle.  We examine $\alpha$-fairness, along with proportional fairness (also known as a the Nash bargaining solution) as a special case, followed by an evaluation of the the Kalai-Smorodinsky bargaining solution.  We then take up the Hooker-Williams scheme as well as a proposal for combining utility with a lexicographic criterion.   

In an effort to address some of the shortcomings of these methods, we next develop our approach to combining leximax fairness and utilitarianism. An earlier version of this scheme is briefly described in a conference paper (\citeauthor{CheHoo20} \citeyear{CheHoo20}).  In Section \ref{sec:defSWF} below, we define SWFs that can be sequentially maximized, subject to the constraints of the given problem, to obtain a socially optimal solution for a specified tradeoff parameter $\Delta$.  We study the mathematical properties of these SWFs in Section~\ref{sec:PD} and describe the sequential optimization procedure more precisely in Section~\ref{sec:optimization}.  A key element of our proposal is a set of practical MILP models for these optimization problems, which we describe and validate in Section~\ref{sec:MILP}.  This is followed in Section~\ref{sec:sharpness} with a family of valid inequalities that can be added to tighten the models.  Real applications often require utility distribution to groups rather than individuals, such as organizations, regions, or demographic groups.  We therefore generalize our method in Section~\ref{sec:group} to deal with groups having specified sizes.    

Section~\ref{sec:guide} provides step-by-step practical guidelines for incorporating our modeling approach into an existing optimization model.  Practitioners may wish to skip directly to this section, which is intended to be as self-contained as possible.  In Section~\ref{sec:exp}, we demonstrate the practical applicability of our approach  by applying it to a healthcare resource allocation problem and an emergency preparedness problem.  The former allows us to compare results with those reported by \citeauthor{HooWil12} (\citeyear{HooWil12}) on the same problem.  The latter is a shelter location and assignment problem of realistic size.  We find that our approach yields reasonable and nuanced socially optimal solutions for both problems, with computation times ranging from a fraction of a second to 18 seconds for a given $\Delta$.  The paper wraps up with concluding remarks and two Appendices containing proofs of theorems that were not proved in the body of the paper.


\section{Schemes for Combining Equity and Efficiency} \label{sec:lit}

We begin by evaluating some of the primary alternatives to our proposed method for balancing equity and efficiency in an optimization model.  For a comprehensive review of equity/efficiency modeling, we refer the reader to Karsu and Morton (2015), who also survey a wide range of applications in which equity measures are used and describe research in the mathematical foundations of equity modeling.

We view each scheme for integrating equity and efficiency as proposing a social welfare function (SWF) $F(\vu)$, where $\vu=(u_1,\ldots,u_n)$ and $u_i$ is the amount of utility allocated to party $i$.  The function $F(\vu)$ appears as the objective function in an optimization model. The model contains resource limitations and other constraints imposed by the problem, and these define a feasible set of utility vectors $\vu$.  The SWF is then maximized over this feasible set.


\subsection{Fairness as Equality}

One approach to combining fairness and efficiency is to use equality as a surrogate for fairness.  This poses two questions: which measure of equality is appropriate, and how should it be combined with efficiency in an SWF?  We address the latter question in the next section.

While it is clear what it means for utilities to be equal, it is much less clear how to measure departures from equality.  A number of metrics have been proposed, including relative mean deviation, variance, coefficient of variation, McLoone index, Gini coefficient, Atkinson index, Hoover index, and the Theil index.  Many of these are discussed and analyzed by \citeauthor{Cow00} (\citeyear{Cow00}) and \citeauthor{JenVanKer11} (\citeyear{JenVanKer11}).  By far the most popular is the Gini coefficient, which is regularly used to gauge income and wealth inequality.  The Gini coefficient measures the relative mean difference between all pairs of utilities and can be defined
\begin{equation}
G(\vu) = \frac{\ds \sum_{i<j}|u_i-u_j|}{\ds n\sum_i u_i}
\label{eq:gini}
\end{equation}
Perfect equality corresponds to a Gini coefficient of 0, and perfect inequality (in which all the utility is allocated to one party) to a coefficient of 1.  One attractive property of the Gini coefficient is that it satisfies the {\em Pigou-Dalton condition}, often viewed as a desirable characteristic for inequality measures.  It states that transferring utility from a better-off party to a worse-off party should increase (or at least not decrease) social welfare.  A practical issue with the Gini coefficient and most other inequality metrics, however, is their nonlinearity, which could complicate solution of an optimization model.  If one wishes to minimize (\ref{eq:gini}) alone, the problem can be linearized using linear-fractional programming (\citeauthor{ChaCoo62} \citeyear{ChaCoo62}), but nonlinearity remains an issue when (\ref{eq:gini}) is combined with an efficiency measure.

A fundamental weakness of any equality-based surrogate for fairness is that equality is not the same concept as fairness (\citeauthor{Fra15} \citeyear{Fra15}).  While excessive inequality is  undesirable in a number of contexts (\citeauthor{Sca03} \citeyear{Sca03}), minimizing inequality is arguably not synonymous with maximizing fairness, because it gives no particular priority to the less advantaged (\citeauthor{Par97} \citeyear{Par97}).  It can equate a utility-neutral rearrangement of wealth that promotes equality among millionaires with a utility-neutral rearrangement that improves the lot of the very poor.  It is therefore at odds with the basic intuition we have adopted for this study, namely that fairness implies giving priority to those who need it most.   

\subsection{Convex Combinations}
We now move to the question of how a fairness measure can be combined with an efficiency measure. The most obvious alternative is to use a convex combination:
\[
F(\vu) = (1-\lambda)\sum_i u_i + \lambda \Phi(\vu)
\]
where $\Phi(\vu)$ is a fairness measure.  A perennial problem with convex combinations is that it is difficult to interpret $\lambda$, particularly when $\Phi(\vu)$ is measured in units other than utility.  For example, if we use the Gini coefficient $G(\vu)$ as a measure of inequity, then we must combine utility with a dimensionless quantity $\Phi(\vu)=1-G(\vu)$.  Larger values of $\lambda$ give greater weight to equality, but in a practical situation it is unclear how to attribute any meaning to a chosen value of $\lambda$. 

\citeauthor{EisTzu19} (\citeyear{EisTzu19}) use a product rather than a convex combination of utility and $1-G(\vu)$, which nicely reduces to an SWF that is easily linearized: 
\[
F(\vu) = \sum_i u_i - \frac{1}{n}\sum_{i<j} |u_j-u_i|
\]
Yet we now have a convex combination of total utility and another equality metric (one that is proportional to the negative mean absolute difference); in particular, it is a convex combination in which $\lambda=1/2$.  This may be reasonable for the intended application, but one may ask why this particular value of $\lambda$ is suitable, and whether other values should be used in other contexts.  Aside from this are the general issues raised by using equality as a surrogate for fairness.

One can combine utility with the Rawlsian maximin criterion by using the convex combination
\begin{equation}
F(\vu) = (1-\lambda)\sum_i u_i  + \lambda \min_i\{u_i\}\
\label{eq:convexComb1}
\end{equation}
This, at least, combines quantities that are measured in the same units.  Yet it is again unclear how to select a suitable value of $\lambda$.   Note that if we index utilities so that $u_1\leq\cdots\leq u_n$, (\ref{eq:convexComb1}) is simply a weighted sum $u_1 + (1-\lambda)\sum_{i>1}u_i$ that gives somewhat more weight to the lowest utility.  Yet how much more is appropriate?  
One can refine criterion (\ref{eq:convexComb1}) by giving gradually decreasing weights $w_1 > w_2 > \cdots > w_n$ to the utilities in an SWF of the form
\begin{equation}
F(\vu) = \sum_i w_iu_i  
\label{eq:convexComb2}
\end{equation}
Yet this only complicates the task of assigning weights.  In addition, since we do not know how to index the utilities by size in advance, we have the difficult modeling challenge of ensuring that weight $w_i$ is assigned to the $i$th smallest utility. 

Finally, one might attempt to combine efficiency with a leximax criterion by using suitable weights $w_i$ in (\ref{eq:convexComb2}).  \citeauthor{Ogryczak2006} (\citeyear{Ogryczak2006}) show which weights accomplish this for a purely leximax objective, but the weights tend to vary enormously in size, which introduces numerical instability.  In addition, it is unclear how to adjust the weights so as to incorporate efficiency into the objective function.

\subsection{$\alpha$-Fairness and Proportional Fairness}

A much-discussed family of SWFs known as $\alpha$-fairness have the form
\[
F_{\alpha}(\vu) = \left\{
\begin{array}{ll}
{\ds
\frac{1}{1-\alpha}\sum_i u_i^{1-\alpha}
} & \mbox{for} \; \alpha\geq 0, \; \alpha\neq 1 \vs \\
{\ds
\sum_i \log(u_i)
} & \mbox{for} \; \alpha=1
\end{array}
\right.
\]
These SWFs form a continuum that stretches from a utilitarian criterion ($\alpha=0$) to a maximin criterion as $\alpha\rightarrow\infty$. They can therefore be seen as blending equity and efficiency, with larger values of $\alpha$ implying a greater emphasis on equity.   \citeauthor{LanKaoChiSab2010} (\citeyear{LanKaoChiSab2010}) provide an axiomatic treatment of $\alpha$-fairness in the context of network resource allocation, and \citeauthor{Bertsimas2012} (\citeyear{Bertsimas2012}) study worst-case equity/efficiency trade-offs implied by this criterion.

The parameter $\alpha$ can be interpreted as quantifying the equity/efficiency trade-off, because utility $u_j$ must be reduced by $(u_j/u_i)^{\alpha}$ units to compensate for a unit increase in $u_i$ ($<u_j$) while maintaining constant social welfare.  This gives priority to less-advantaged parties, as we desire, with $\alpha$ indicating how much priority.  In particular, the fact that $(u_j/u_i)^{\alpha}>1$ when $u_i<u_j$ implies that $F_{\alpha}(\vu)$ satisfies the Pigou-Dalton condition for all $\alpha$.  Yet it is not obvious what kind of trade-off, and therefore what value of $\alpha$, is appropriate for a given application.  The underlying difficulty is that there is no apparent interpretation of $\alpha$ independent of its role in the SWF.  There is also the computational impediment that $F_{\alpha}(\vu)$ is nonlinear.
  
A well-known special case of $\alpha$-fairness arises when $\alpha=1$.  This results in proportional fairness, often measured by the product $\Pi_i u_i$ rather than its logarithm $\sum_i \log(u_i)$.  Maximizing proportional fairness yields the Nash bargaining solution (\citeauthor{Nas50} \citeyear{Nas50}), which should not be confused with the Nash equilibrium of game theory.  It corresponds to selecting a point $\vu$ in the feasible set that maximizes the volume of the hyperrectangle with opposite corners at $\vu$ and the origin.  This is illustrated in Fig.~\ref{fig:NashRKS}(a), where each point on the plot represents the utility outcomes for two parties that result from some distribution of resources.  The set of feasible utility vectors is the area under the curve.  The Nash bargaining solution is the black dot, which is the feasible point that maximizes the area of the shaded rectangle.  Proportional fairness is frequently used in engineering, such as for bandwidth allocation in telecommunication networks and traffic signal timing. 

\begin{figure}[!t]
	\centering
	\begin{tabular}{c@{\hspace{10ex}}c}
	\includegraphics[trim=120 490 320 120, clip, scale=0.9]{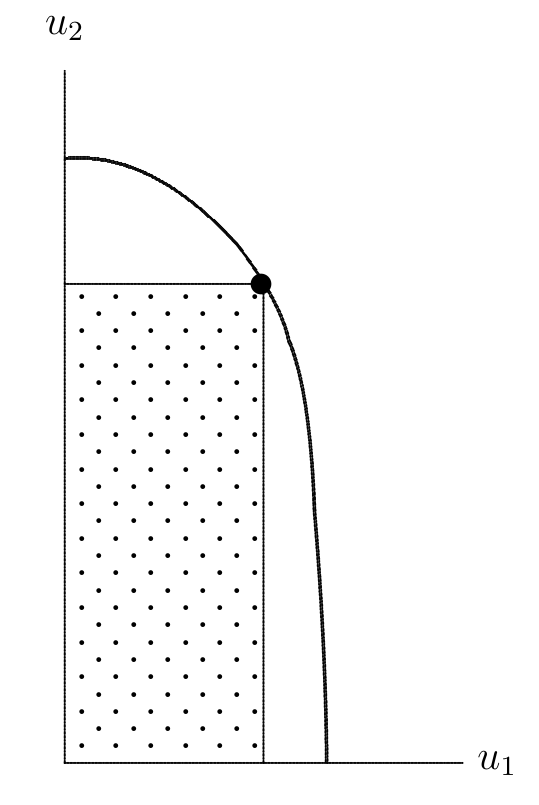} &
	\includegraphics[trim=120 490 320 120, clip, scale=0.9]{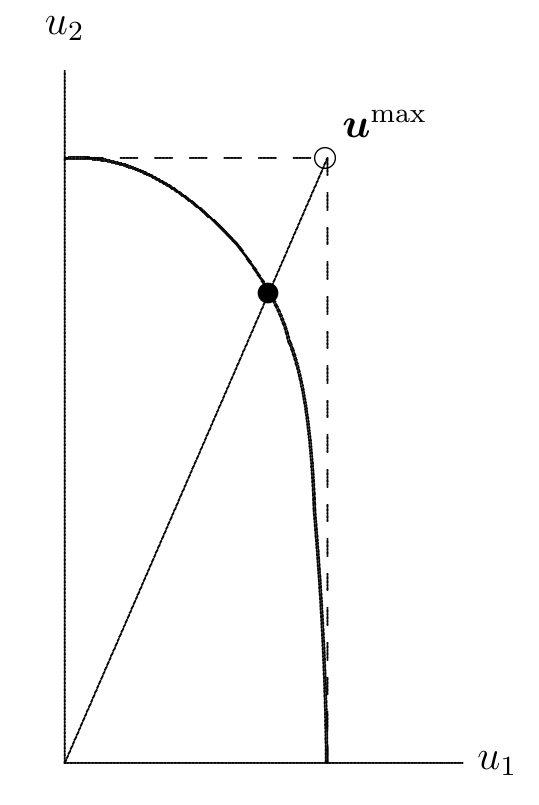} \\
	(a) & (b)
	\end{tabular}
	\caption{(a) Nash bargaining solution for two players.  (b) Kalai-Smorodinsky bargaining solution for two players.  In both cases, the default position is the origin.}
	\label{fig:NashRKS}
\end{figure}

Proportional fairness has axiomatic and bargaining-based derivations that might be seen as justifying the parameter setting $\alpha=1$.  For example, \citeauthor{Nas50} (\citeyear{Nas50}) showed that his bargaining solution for two persons is implied by a set of axioms for utility theory. \citeauthor{Har77} (\citeyear{Har77}), \citeauthor{Rub82} (\citeyear{Rub82}), and \citeauthor{BinRubWol86}\ (\citeyear{BinRubWol86}) showed that the Nash solution is the (asymptotic) outcome of certain rational bargaining procedures.  Yet the axiomatic derivation relies on a strong axiom of cardinal noncomparability across parties that is central to the proof.  The axiom assumes that the ranking of utility vectors is invariant under affine transformations of the form $\phi_i(u_i)=\beta_iu_i + \gamma_i$, which arguably rules out the kind of utility comparisons we need in order to assess fairness (\citeauthor{Hoo13} \citeyear{Hoo13}).  Furthermore, the bargaining theories assume that the parties begin with a default utility allocation $\bm{d}=(d_1,\ldots,d_n)$ on which they fall back if bargaining fails, as illustrated by our medical example in Section~\ref{sec:healthexp}.  The proportional fairness SWF then becomes $F(\vu)=\Pi_i(u_i-d_i)$.  An unfair starting point $\bm{d}$ could lead to an unfair outcome even under a  rational bargaining procedure, and even if we grant that rational bargaining from a fair starting point necessarily yields a fair outcome.  This weakens the bargaining argument for the fairness of the Nash solution in general.

\subsection{Kalai-Smorodinsky Bargaining}

The Kalai-Smorodinsky (K-S) bargaining solution was proposed as an alternative to the Nash bargaining solution (\citeauthor{KalSmo75} \citeyear{KalSmo75}).  It provides players the largest possible increase in utility relative to the maximum increase they could obtain if other players were disregarded, subject to the condition that the relative increase is the same for all players.  Increases in utility are measured with respect to the default utility allocation.  This scheme is sometimes described as minimizing each player's relative concession.

One motivation for the K-S criterion is that it maximizes total utility while maintaining fairness for all players, where fairness takes into account the fact that allocating utility to some players is more costly than to others.  The approach is defended by \citeauthor{Tho94} (\citeyear{Tho94}) and is arguably consistent with the contractarian ethical philosophy developed by \citeauthor{Gau87} (\citeyear{Gau87}).   

Mathematically, the objective is to find the largest scalar $\beta$ such that $\vu=(1-\beta)\bm{d} + \beta \vu^{\max}$ is a feasible utility vector, where each $u_i^{\max}$ is the maximum of $u_i$ over all feasible utility vectors $\vu$.  The bargaining solution is the vector $\vu$ that maximizes $\beta$.  Figure~\ref{fig:NashRKS}(b) illustrates the idea for two players when the default position $\bm{d}$ is the origin.  The K-S solution (black dot) is the highest point at which the diagonal line intersects the feasible set. 

Axiomatic justifications are given for the K-S solution by Kalai and Smorodinsky as well as by Thompson, but they again rely on a strong axiom of cardinal noncomparability.  A bargaining justification might be given by arguing that it is rational for each player to minimize relative concession, and repeated rounds of bargaining will lead under suitable conditions to an equilibrium in which their relative concessions are equal and minimized.  

Yet the K-S criterion leads to an anomalous situation when one player's utility gain is very expensive and therefore requires very large utility transfers from the other players (\citeauthor{Hoo13} \citeyear{Hoo13}).  This forces the overall utility gain to be very small in the K-S solution when much greater utility gains are possible.  In the extreme case where one player's utility cannot be increased at all unless other utilities are held constant, the K-S solution obliges all players to settle for the default position with no increase in utility.  That is, if $u_i^{\max}>d_i$ and $\max\{u_i\;|\; \vu\in S, \;u_j>0 \;\mbox{for $j\neq i$}\}=d_i$, where $S$ is the feasible set, then the K-S solution is $\beta=0$ and $\vu=\bm{d}$.  This wastes all potential increases in total utility simply because one player cannot benefit when others benefit.  This phenomenon does not occur in the other schemes considered here, in which there are socially optimal solutions that increase utility in a scenario of this kind.  

Finally, there is apparently no generalization of the K-S scheme, analogous to $\alpha$-fairness and Nash bargaining, that allows parameterization of the equity/efficiency trade-off.

\subsection{Combining Utilitarian and Rawlsian Criteria}

We now examine two additional methods in the literature for combining utility with a maximin or leximax criterion.  One is the Hooker-Williams SWF described earlier, which combines utility with a maximin criterion.  It was inspired by a 2-person SWF put forward by \citeauthor{WilliamsCookson2000} (\citeyear{WilliamsCookson2000}).  The indifference curves (contours) of the SWF are illustrated in Fig.~\ref{fig:2person}.  
The function is utilitarian when $|u_1-u_2|\geq \Delta$ and represents a maximin criterion otherwise.  Specifically,
\[
F(u_1,u_2) = \left\{ \begin{array}{ll} u_1 + u_2, & \mbox{if $|u_1-u_2|\geq \Delta $} \\
2 \min\{u_1,u_2\}+\Delta, & \mbox{otherwise}
\end{array} \right.
\]
The maximin criterion would ordinarily be $\min\{u_1,u_2\}$, but it is modified here to obtain continuous contours as one moves from the utilitarian to the maximin objective.  

The feasible set in Fig.~\ref{fig:2person} is again the portion of the nonnegative quadrant under the curve.  In a medical context, for example, it represents all feasible health outcomes that are within the resource budget.  The shape of the curve indicates that when patient 1's health reaches a certain point, further improvement requires extraordinary sacrifice by patient 2 due to the transfer of resources.  The utilitarian solution (black dot in the figure) might therefore be viewed as preferable to the maximin solution (small open circle) and in fact yields slightly more social welfare as indicated by the contours. 

\begin{figure}
	\centering
	\includegraphics[trim=70 475 220 120, clip, scale=0.9]{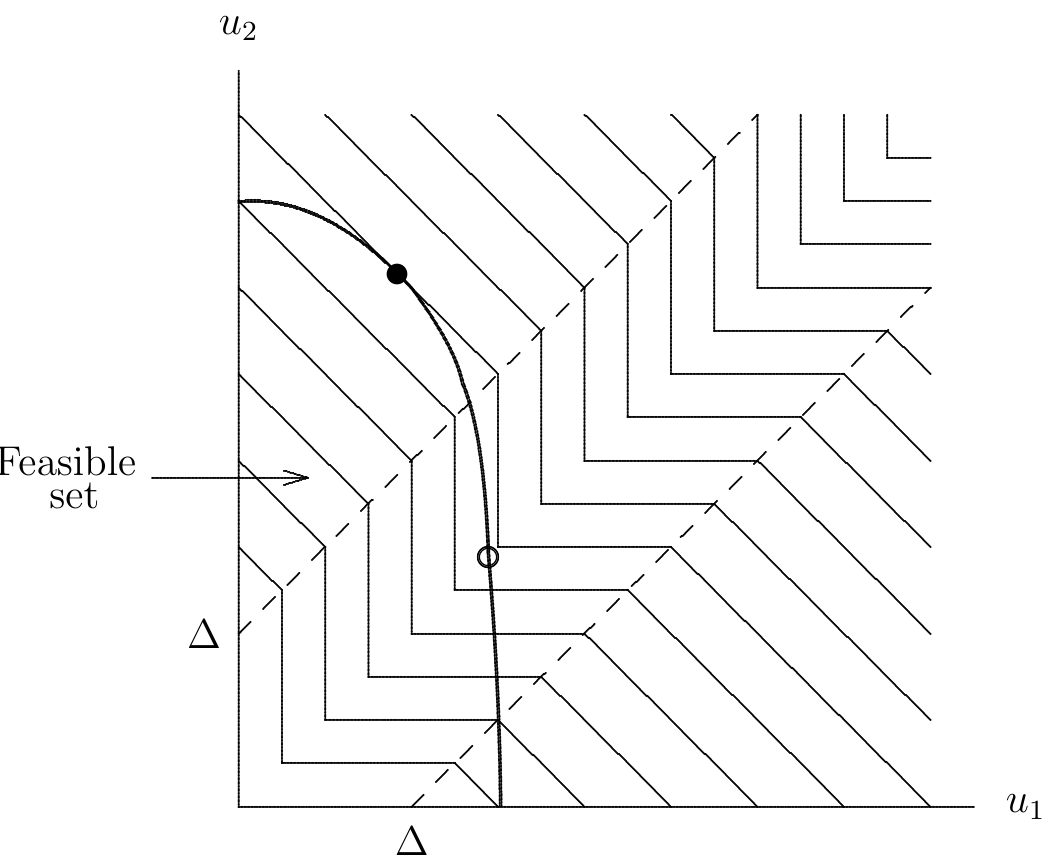}
	\caption{Piecewise linear social welfare contours for 2 persons.}
	\label{fig:2person}
\end{figure}

\citeauthor{HooWil12} (\citeyear{HooWil12}) extend this social welfare function to $n$ persons as follows.  They observe that it can be written
\[
F(u_1,u_2) = \Delta + 2u_{\langle 1\rangle} + (u_1-u_{\langle 1\rangle}-\Delta)^+ + (u_2-u_{\langle 1\rangle}-\Delta)^+
\]
where $(\alpha)^+=\max\{0,\alpha\}$, and where we adopt the convention that $(u_{\langle 1\rangle},\ldots, u_{\langle n\rangle})$ is the tuple $(u_1,\ldots, u_n)$ arranged in non-decreasing order.  This function is readily generalized to $n$ persons:
\begin{equation}
F_1(\vu) = (n-1)\Delta + nu_{\langle 1\rangle} + \sum_{i=1}^n (u_i-u_{\langle 1\rangle}-\Delta)^+ 
\label{eq:swf1}
\end{equation}
We refer to the function as $F_1$ because it will be the first in a series of functions $F_1,\ldots,F_n$ we define later.  It may be more intuitive to rewrite (\ref{eq:swf1}) as
\[
F_1(\vu) = \big(t(\vu)-1\big) \Delta + \sum_{i=1}^{t(\vec{u})} u_{\langle 1\rangle} + \hspace{-2ex} \sum_{i=t(\vec{u})+1}^n \hspace{-2.5ex} u_{\langle i\rangle}
\]
where $t(\vu)$ is defined so that $u_{\langle 1\rangle},\ldots, u_{\langle t(\vu)\rangle}$ are within $\Delta$ of $u_{\langle 1\rangle}$; that is, $u_{\langle i\rangle} - u_{\langle 1\rangle} \leq \Delta$ if and only if $i\leq t(\vu)$.  We will refer to utilities $u_{\langle 1\rangle}, \ldots, u_{\langle t(\vu)\rangle}$ as being {\em in the fair region} and utilities $u_{\langle t(\vu)+1\rangle}, \ldots, u_{\langle n\rangle}$ as being {\em in the utilitarian region}.  
The function $F_1(\vu)$ therefore has the effect of summing all the utilities, but with the proviso that utilities in the fair region are counted as equal to $u_{\langle 1\rangle}$.  Thus the relatively disadvantaged individuals (those within $\Delta$ of the worst-off) are treated in solidarity with the worst-off, in the sense that their lot is equated with that of the worst-off.  The term $(t(\vu)-1)\Delta$ is added to ensure continuity of the function.

The parameter $\Delta$ therefore has an interpretation that can be described independently of its role in the SWF.  Namely, any party with utility within $\Delta$ of the lowest is viewed as disadvantaged and deserving of special consideration.  The SWF then defines the special consideration to be an identification of the disadvantaged party with the worst-off party, which is given disproportionate weight in the summation of utilities---namely, weight equal to the number of utilities within $\Delta$ of the lowest.  


\citeauthor{GerKanCla18}\ (\citeyear{GerKanCla18}) make some interesting observations regarding properties of the SWF (\ref{eq:swf1}).  In particular, the solutions obtained by varying $\Delta$ need not all lie on the Pareto frontier defined by the convex combination (\ref{eq:convexComb1}) of utilitarian and maximin objectives.  This is in fact to be expected, because the convex combination balances total utility with only the welfare of the worst-off party, while (\ref{eq:swf1}) balances total utility with the welfare of all disadvantaged parties.

A problem with (\ref{eq:swf1}), however, is that the actual utility levels of the disadvantaged parties, other than that of the very worst-off, have no effect on the value of the SWF.  This is illustrated in the 3-person example of Fig.~\ref{fig:3person}, which shows the contours of $F(u_1,u_2,u_3)$ with $\Delta=3$ and $u_1$ fixed to zero.  The SWF is constant in the shaded region, meaning that the utilities allocated to persons 2 and 3 have no effect on social welfare as measured by $F_1(\vu)$, so long as they remain in the fair region.  As a result, there are infinitely many utility vectors that maximize social welfare, some of which differ greatly with respect to utilities in the fair region.  One can add a tie-breaking term $\epsilon(u_2+u_3)$ to the social welfare function, where $\epsilon>0$ is small, so as to maximize utility as a secondary objective.  Yet this still does not account for equity considerations within the fair region.  

\begin{figure}
	\centering
	\includegraphics[trim=70 500 250 120, clip, scale=0.9]{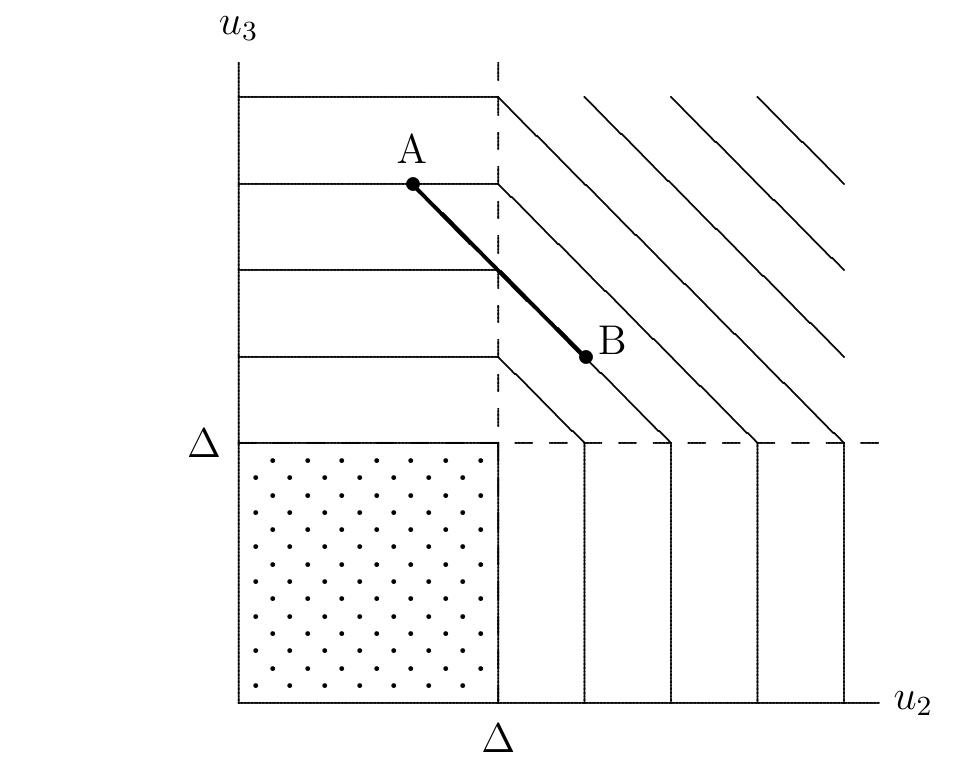}
	\caption{Contours of $F_1(0,u_2,u_3)$.  The function is constant in the shaded region.}
	\label{fig:3person}
\end{figure}

To obtain a SWF that is sensitive to the actual utility levels of all the disadvantaged parties, one might combine utility with a leximax criterion rather than a maximin criterion.  \citeauthor{McElfresh2018} (\citeyear{McElfresh2018}) propose one method of doing so in the context of kidney exchange.  Their method is related to the Hooker-Williams approach, but it relies on the assumption that the parties can be given a preference ordering in advance.  It first maximizes a SWF that combines utilitarian and maximin criteria in a way that treats the most-preferred party as the worst-off.   If all optimal solutions of this problem lie in the utilitarian region, a utilitarian criterion is used to select one of the optimal solutions.  (Here, a utility vector $\vu$ is said to be in the fair region if $\max_i\{u_i\}-\min_i\{u_i\}\leq \Delta$, and otherwise in the utilitarian region.)  Otherwise a leximax criterion is used for all of the optimal solutions, subject to the preference ordering (i.e., maximize $u_1$ first, then $u_2$ etc.).  If we index the parties in order of decreasing preference, the SWF is
\begin{equation}
F(\vu) = \left\{
\begin{array}{ll}
nu_1, & \mbox{if}\; |u_i-u_j|\leq\Delta \;\mbox{for all} \; i,j \vs\\
{\ds
\sum_i u_i + (N^+-N^-)\Delta, 
} & \mbox{otherwise}
\end{array}
\right.   \label{eq:kidney}
\end{equation}
where $N^+=|\{i\;|\;u_1>u_i\}|$ and $N^-=|\{i\;|\;u_1<u_i\}|$ and the term $(N^+-N^-)\Delta$ achieves continuity.  While it is possible to pre-specify a preference ranking of parties in some applications, such as the kidney exchange problem, this is not possible in many applications.  Also the leximax criterion is not used until optimal solutions of the SWF are already obtained, and then applied only to the optimal solutions.  We prefer to allow a leximax criterion to play a role in evaluating all the possible solutions.

\section{Defining the Social Welfare Functions} \label{sec:defSWF}

We now propose general-purpose SWFs that combine utilitarian and leximax criteria.  Our aim is to preserve the desirable properties of the Hooker-Williams SWF while avoiding some of the shortcomings of it and other SWFs discussed above.  In particular, we regulate the fairness/efficiency trade-off with a parameter $\Delta$ that has a practical interpretation.  By replacing the maximin with a leximax criterion, we allow the fairness criterion to reflect the actual utility levels of all disadvantaged parties rather than just that of the worst-off.  In addition, we do not assume a predefined preference ordering of the parties when applying the leximax criterion.  Thus a utility vector $\vu$ is lexicographically greater than or equal to $\bm{u}'$ when $u_{\langle k \rangle}\geq u'_{\langle k\rangle}$ and $(u_{\langle 1\rangle},\ldots,u_{\langle k-1\rangle})=(u'_{\langle 1\rangle},\ldots,u'_{\langle k-1\rangle})$ for some $k\in \{1,\ldots,n\}$.  

To combine the leximax and utilitarian criteria, we propose optimizing a sequence of social welfare functions $F_1(\vu),\ldots,F_n(\vu)$, each of which combines maximin and utilitarian measures.  The first function $F_1(\vu)$ is the Hooker-Williams function (\ref{eq:swf1}) defined earlier and is maximized over $\vu=(u_1,\ldots,u_n)$ to obtain a value for $u_{\langle 1\rangle}$.  Each subsequent function $F_k(\vu)$ is maximized over $u_{\langle k\rangle},\ldots,u_{\langle n\rangle}$, while fixing utilities $u_{\langle 1\rangle},\ldots,u_{\langle k-1\rangle}$ to the values already obtained, to determine a value for $u_{\langle k\rangle}$.  The process terminates when maximizing $F_k(\vu)$ yields a value of $u_{\langle k\rangle}$ that lies outside the fair region.  At this point, $F_k(\vu)$ is utilitarian, and utilities $u_{\langle k\rangle},\ldots, u_{\langle n \rangle}$ are determined simultaneously by maximizing $F_k(\vu)$ while fixing $u_{\langle 1\rangle},\ldots, u_{\langle k-1\rangle}$ to the values already obtained.  We refer to a utility vector $(u_{\langle 1\rangle},\ldots,u_{\langle n\rangle})$ that results from this process as {\em socially optimal}.
	
We will describe this sequential optimization procedure more precisely in Section~\ref{sec:optimization}, but we must first define and explain the functions $F_k(\vu)$.  To develop functions that are continuous as well as analogous to the Hooker-Williams SWF  $F_1(\vu)$ in (\ref{eq:swf1}), it is helpful to write $F_1(\vu)$ as
\[
F_1(\vu) = t(\vu)u_{\langle 1\rangle} + \big(t(\vu)-1\big)\Delta + \hspace{-2ex} \sum_{i=t(\vu)+1}^n \hspace{-2ex} u_{\langle i\rangle}
\]
Maximizing $F_1(\vu)$ determines the location of the fair region because it determines the value of $u_{\langle 1\rangle}$, and the fair region is the interval $[u_{\langle 1\rangle}, u_{\langle 1\rangle}+\Delta]$.  The function $F_1(\vu)$ gives priority to the smallest utility $u_{\langle 1\rangle}$ by allowing it to represent all the utilities in the fair region and assigning it a weight $t(\vu)$ equal to the size of the fair region, while utilities in the utilitarian region receive only unit weight.   However, utilities in the fair region other than $u_{\langle 1\rangle}$ do not appear in the SWF, and their actual values therefore have no bearing on social welfare.  One can even reduce them to $u_{\langle 1\rangle}$ without affecting the value of $F_1(\vu)$. 

We now generalize the $F_1(\vu)$ to $F_k(\vu)$ in a way that takes into account the other utilities in the fair region.  For $k=2,\ldots,n$, we define
\begin{equation}
F_k(\vu) = \left\{
\begin{array}{ll}
{\ds
\sum_{i=1}^k (n-i+1) u_{\langle i\rangle} 
+ \hspace{-2ex} \sum_{i=t(\vu)+1}^n \hspace{-2ex} (u_{\langle i\rangle}-u_{\langle 1\rangle}-\Delta),
} & \mbox{if $t(\vu)\geq k$} \vs \vs \\
{\ds
\sum_{i=1}^n u_{\langle i\rangle}, 
} & \mbox{if $t(\vu)<k$}
\end{array}
\right.
\label{eq:swf2}
\end{equation}
This results in an optimization problem that is quite different than when $k=1$, because the fair region has already been determined by maximizing $F_1(\vu)$.  Utilities $u_{\langle 1\rangle},\ldots, u_{\langle k-1\rangle}$ have been fixed by the solutions of prior optimization problems and therefore appear in $F_k(\vu)$ as constants, while $u_{\langle k\rangle}$ remains as the only variable in the fair region.  The actual value of this utility plays a role in the SWF, since increasing $u_{\langle k\rangle}$ increases $F_k(\vu)$.  Furthermore, $u_{\langle k\rangle}$ receives more weight than utilities in the utilitarian region, because it has weight $n-k+1$, while those in the utilitarian region receive weight 1.  Since the weight is greater for small $k$, less advantaged parties receive greater priority.  However, the weight does not depend on the number of utilities in the fair region as in $F_1(\vu)$, as this would result in a discontinuous SWF.  


Thus as $F_k(\vu)$ is maximized for increasing values of $k$, each utility $u_{\langle k\rangle}$ in the fair region receives priority at some point in the process.  The process stops when maximizing $F_k(\vu)$ yields a solution in which $u_{\langle k\rangle}$ is the only remaining utility in the fair region, at which point the solution determines the values of all remaining utilities $u_{\langle k\rangle},\ldots,u_{\langle n\rangle}$.  This scheme incorporates lexicographic optimization in the sense that the smaller utilities are determined earlier in the sequence, although rather than maximizing $u_{\langle k\rangle}$ in step $k$, we maximize a SWF that gives priority to $u_{\langle k\rangle}$.  Utilitarianism in incorporated because each maximization problem considers total utility as well as fairness.

For extreme values of $\Delta$, this process yields purely utilitarian or purely leximax solutions.  When $\Delta=0$, we have $t(\vu)=1$ for all $\vu$, and $F_1(\vu)$ reduces to a utilitarian criterion.  The fair region is the single point $u_{\langle 1\rangle}$, and we solve the social welfare problem simply by maximizing $F_1(\vu)$, which yields a utilitarian solution.  For sufficiently large $\Delta$, $t(\vu)=n$ for all feasible $\vu$, and $F_k(\vu)$ is $(n-k+1)u_{\langle k\rangle}$ plus a constant for $k=1,\ldots, n$.  Since all $u_i$s lie in the fair region, we sequentially maximize $F_k(\vu)$ for $k=1,\ldots n$ and therefore obtain a pure leximax solution.  Intermediate values of $\Delta$ combine utilitarian and leximax criteria.


Figure~\ref{fig:3personNew} illustrates how maximizing $F_1(\vu),\ldots,F_n(\vu)$ sequentially is more sensitive to equity than maximizing $F_1(\vu)$, which has the flat region shown in Fig.~\ref{fig:3person}, as noted earlier.  Suppose we determine a value for $u_1$ by maximizing $F_1(\vu)$, say $u_1=0$.  Then the function $F_2(\vu)$ has no flat regions, as is evident in Fig.~\ref{fig:3personNew}, and the solutions in the flat region of Fig.~\ref{fig:3person} are now distinguished.  
	
\begin{figure}
\centering
\includegraphics[trim=70 450 250 120, clip, scale=0.9]{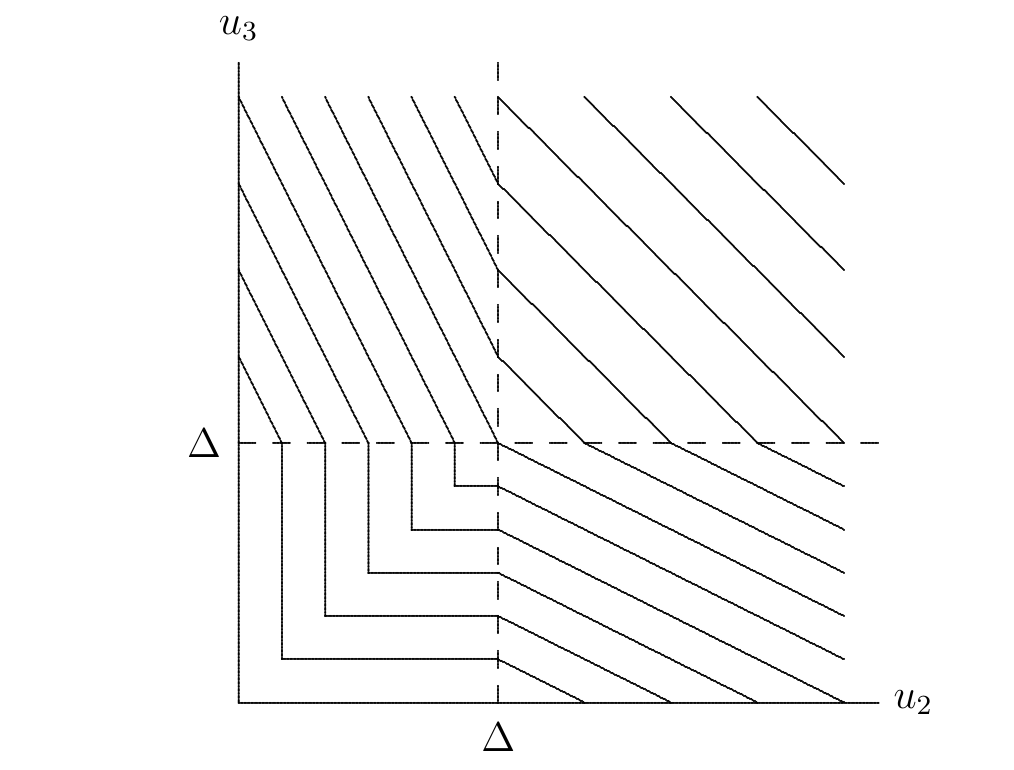}
\caption{Contours of $F_2(0,u_2,u_3)$ with $\Delta=3$ and contour interval 1.}
\label{fig:3personNew}
\end{figure}


\section{Properties of the Social Welfare Functions} \label{sec:PD}

We now investigate some mathematical properties of the social welfare functions $F_k(\vu)$.  After proving continuity, we show that while $F_k(\vu)$ need not satisfy the Pigou-Dalton condition, it satisfies a somewhat weaker condition.  

\begin{theorem} \label{thm:continuity}
The functions $F_k(\vu)$ are continuous for $k=1,\ldots, n$.
\end{theorem}

\begin{proof}
To prove continuity of $F_1(\vu)$, it suffices to show that each term of (\ref{eq:swf1}) is continuous, because a sum of continuous functions is continuous.  The first term of (\ref{eq:swf1}) is a constant, and the second term is continuous because order statistics are continuous functions.  Each term of of the summation is continuous because it is the maximum of two continuous functions.  To show that $F_k(\vu)$ is continuous for $k\geq 2$, it is convenient to write (\ref{eq:swf2}) as
\begin{equation}
F_k(\vu) = \sum_{i=1}^{k-1} (n-i+1)u_{\langle i\rangle} 
+ (n-k+1) u_{\langle k\rangle} 
- (n-k)(u_{\langle k\rangle} - u_{\langle 1\rangle} - \Delta)^+
+ \hspace{-0.5ex} \sum_{i=k+1}^n \hspace{-0.5ex} (u_{\langle i\rangle}-u_{\langle 1\rangle}-\Delta)^+
\label{eq:swf2a}
\end{equation}
which simplifies to
\[
F_k(\vu) = \sum_{i=1}^{k-1} (n-i+1) u_{\langle i\rangle} 
+ (n-k+1) \min\{u_{\langle 1\rangle} + \Delta, u_{\langle k\rangle}\}
+ \hspace{-0.5ex} \sum_{i=k}^n (u_{\langle i\rangle}-u_{\langle 1\rangle}-\Delta)^+
\]
Because order statistics are continuous, $u_{\langle k\rangle}$ and $u_{\langle 1\rangle}$ are continuous functions of $\vu$.  Also $\min\{u_{\langle 1\rangle} + \Delta, u_{\langle k\rangle}\}$ and $(u_{\langle i\rangle} - u_{\langle 1\rangle}-\Delta)^+$ are continuous because they are the minimum or maximum of continuous functions.
\end{proof}

We now investigate whether the SWFs $F_k(\vu)$ satisfy the Pigou-Dalton and related conditions.  Recall that the Pigou-Dalton condition requires that any utility transfer from a better-off party to a worse-off party increases (or at least does not decrease) social welfare.  We will not require a strict increase, because otherwise a pure utilitarian SWF fails the condition.  
Thus $F(\vu)$ satisfies the Pigou-Dalton condition 
if $F(\vu + \epsilon\vec{e}_i - \epsilon\vec{e}_j)\geq F(\vu)$ for any $i,j$ and any $\epsilon>0$, where $u_i<u_j$ and $\vec{e}_i$, $\vec{e}_j$ are the $i$th and $j$th unit vectors, respectively.  
	
It is stated in \citeauthor{Karsu2015} (\citeyear{Karsu2015}) that the Hooker-Williams SWF $F_1(\vu)$ satisfies the Pigou-Dalton condition, but this is true only for $n=2$.  Figure~\ref{fig:3person} provides a counterexample for $n=3$.  The move from point $A$ to point $B$ represents a utility transfer from a better-off party to a worse-off party.  This strictly reduces social welfare, because $B$ lies on a lower contour than $A$.  Similar counterexamples show that $F_k(\vu)$ can violate the Pigou-Dalton condition for $k\geq 2$.

While it is widely recognized that an equality measure should satisfy the Pigou-Dalton condition, this is not obviously true of $F_k(\vu)$, since it is not an equality measure or even a fairness measure.  In any event,
\citeauthor{ChaMoy05} (\citeyear{ChaMoy05}) have defined a weaker form of the Pigou-Dalton condition that all the functions $F_k(\vu)$ satisfy.  It is based on transfers of utility from a better-off class to a worse-off class rather than from one individual to another.  Specifically, it examines the consequences of transferring a given amount of utility from individuals whose utility lies above any given threshold (taking an equal share from each) to those whose utility lies below any given threshold (giving an equal share to each).  Arguably, only such transfers should be considered, because a removal of utility from the upper range should remove at least as much from the best-off individual as from other well-off individuals, and an endowment of utility on the lower range should benefit the worst-off individual at least as much as other badly-off individuals.  
The Chateauneuf-Moyes (C-M) condition requires that such transfers result in at least as much social welfare.
	
To define the C-M condition formally, let us say that a {\em C-M transfer} is a transfer of utility from $\vu$ to $\vu'$ such that $u_1\leq \cdots \leq u_n$ as well as $u'_1\leq\cdots\leq u'_n$, and for some pair of integers $\ell$, $h$ with $1\leq \ell < h \leq n$, we have $u_{\ell}<u_h$ and 
\[ 
\vu' = \vu + \frac{\epsilon}{\ell}\sum_{i=1}^{\ell}\bm{e}_i - \frac{\epsilon}{n-h+1}\sum_{i=h}^n \bm{e}_i
\]
A SWF $F(\vu)$ satisfies the C-M condition if C-M transfers never decrease social welfare.  That is, for any $\vu$ and any C-M transfer from $\vu$ to $\vu'$,
\begin{equation}
F(\vu') \geq F(\vu)
\label{eq:C-M}
\end{equation}
for sufficiently small $\epsilon>0$.

\begin{theorem} \label{th:CM1}
	The Hooker-Williams social welfare function $F_1(\vu)$ satisfies the Chateauneuf-Moyes condition.
\end{theorem}
	
\begin{proof}
It suffices to show that (\ref{eq:C-M}) holds for any $\vu$ and sufficiently small $\epsilon>0$.  There are three types of utility transfer, illustrated in Fig.~\ref{fig:CM1}:
(a) $\ell<h\leq t(\vu)$, 
(b) $\ell\leq t(\vu) <h$, and
(c) $t(\vu) <\ell< h$.
The resulting utility gain by individuals $1,\ldots\ell$, and loss by individuals $h,\ldots,n$, are indicated in Table~\ref{ta:CM1}.  It is clear on inspection of Fig.~\ref{fig:CM1} that the gain is at least $\epsilon$ in each case, and the loss never more than $\epsilon$.  The C-M condition is therefore satisfied.
\end{proof}
	
\begin{figure}[!t]
	\centering
	\includegraphics[trim=130 640 120 90, clip]{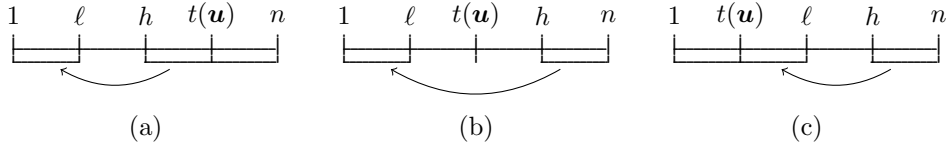}
	\caption{Illustration of proof of Theorem \ref{th:CM1}.}
	\label{fig:CM1}
\end{figure}
	
\begin{table}[!h]
	\centering
	\caption{Verifying the Chateauneuf-Moyes condition for $F_1(\vu)$} \label{ta:CM1}
	\vs\vs
	\begin{tabular}{ccc}
		Case & Gain & Loss \\  
		\hline \\ [-2ex]
		(a) & ${\ds \frac{t(\vu)}{\ell}\epsilon >\epsilon}$  & ${\ds \frac{n-t(\vu)}{n-h+1}\epsilon<\epsilon}$ \vs \vs\\
		(b) & ${\ds \frac{t(\vu)}{\ell}\epsilon > \epsilon}$ & $\epsilon$ \vs \vs \\
		(c) & $\epsilon$ & $\epsilon$ \\ 
		\hline
\end{tabular}
\end{table}

A similar proof (with six cases), given in Appendix~1, shows that the remaining SWFs satisfy the C-M condition.

\begin{theorem} \label{th:CM2}
	The social welfare functions $F_k(\vu)$ satisfy the Chateauneuf-Moyes condition for $k=2,\ldots,n$.
\end{theorem}
	
Although each $F_k(\vu)$ satisfies the C-M condition, the feasible set can be designed in such a way that a C-M transfer may not transform a socially optimal solution to another socially optimal solution.  We will demonstrate this in the next section.  

\section{The Sequential Optimization Procedure}  \label{sec:optimization}
	
We now describe in detail how one can obtain a socially optimal utility distribution.
We wish to maximize the sequence of social welfare functions $F_1(\vu),\ldots,F_n(\vu)$, subject to resource constraints, in such a way that maximizing $F_k(\vu)$ determines the value of the $k$th smallest $u_i$ in the socially optimal solution.  We therefore maximize $F_k(\vu)$ subject to the condition that (a) the $u_i$s already determined are fixed to the values obtained for them, and (b) the remaining $u_i$s must be at least as large as the largest $u_i$ already determined.  The unfixed $u_i$ with the smallest value in the solution becomes the utility determined by maximizing $F_k(\vu)$.
	
For now, we indicate resource limits by writing $\vu\in\mathcal{U}$.  In practice, they would be formulated in a MILP model by introducing variables and constraints that specify resource limitations and how resource allocations to individual parties translate to utilities.  This will be illustrated in our experiments in Section \ref{sec:exp}. 
	
To state the optimization procedure more precisely, we define a sequence of maximization problems $P_1,\ldots,P_k$, where $P_1$ maximizes $F_1(\vu)$ subject to $\vu\in\mathcal{U}$, and $P_k$ for $k=2,\ldots,n$ is
\begin{equation}
\begin{array}{l}
\max \; F_k(\vu) \vs \\
u_{i_j}=\bar{u}_{i_j}, \;\; j=1, \ldots, k-1 \vs \\
u_i \geq \bar{u}_{i_{k-1}}, \;\; i\in I_k \vs \\
\vu \in \mathcal{U}
\end{array}
\label{eq:seq}
\end{equation}
The indices $i_j$ are defined so that $u_{i_j}$ is the utility determined by solving $P_j$.  In particular, $u_{i_j}$ is the utility with the smallest value among the unfixed utilities in an optimal solution obtained by solving $P_j$.  Thus
\[
i_j = \mathop{\mathrm{argmin}}_{i\in I_j} \{u_i^{[j]}\}
\]
where $\vu^{[j]}$ is an optimal solution of $P_j$ and $I_j = \{1,\ldots,n\} \setminus \{i_1,\ldots,i_{j-1}\}$.  We denote by $\bar{u}_{i_j}=u^{[j]}_{i_j}$ the solution value obtained for $u_{i_j}$ in $P_j$.  We need only solve $P_k$ for $k=1,\ldots,K+1$, where $K$ is the largest $k$ for which $\bar{u}_{i_k}\leq \bar{u}_{i_1} + \Delta$.  The solution of the social welfare problem is then
\[
u_i = \left\{
\begin{array}{ll}
\bar{u}_i & \mbox{for} \; i=i_1,\ldots,i_{K-1} \vs \\
u^{[K]}_i & \mbox{for} \; i\in I_{K}
\end{array}
\right.
\]
	
As remarked earlier, a C-M transfer applied to a socially optimal solution need not yield another socially optimal solution.  
To see this, suppose $n=4$, $\Delta=5$, and the feasible set consists only of the three vectors on the left:
\[
\begin{array}{ll}
\vu^1 = (1,2,8,9)  &  (24,15,27,35) \\
\vu^2 = (2,3,7,8)  &  (24,18,32,39) \\
\vu^3 = (1,2,3,12) &  (25,16,22,28) 
\end{array}
\]
The corresponding values $(F_1(\vu),\ldots,F_4(\vu))$ are shown on the right. Distribution $\vu^2$ results from applying a C-M transfer to the socially optimal distribution $\vu^1$, but $\vu^2$ is not socially optimal because $u_1=2$ in no optimal solution of $P_1$. Rather, the unique optimal solution of $P_1$ is $\vu^3$, in which $u_1=1$.  This situation can occur when a socially optimal distribution $\vu$ is not an optimal solution of $P_1$.  In the example, $\vu^1$ is not an optimal solution of $P_1$.

We can simplify the development to follow by removing the initial constants from $F_k(\vu)$ for $k\geq 2$, resulting in the SWF
\begin{equation}
\bar{F}_k(\vu) = (n-k+1)u_{\langle k\rangle}
+ \sum_{i=k}^n (u_{\langle i\rangle} - u_{\langle 1\rangle} - \Delta)^+
\label{eq:swfSimple}
\end{equation}
Although the functions $\bar{F}_k(\vu)$ do not satisfy the C-M condition, they obviously yield the same socially optimal solutions as the original SWFs.  We therefore solve problems $\bar{P}_k$ rather than $P_k$ for $k\geq 2$, where $\bar{P}_k$ results from replacing $F_k(\vu)$ with $\bar{F}_k(\vu)$ in $P_k$.  For notational convenience we let $\bar{P}_1$ denote the original problem $P_1$.

\section{Mixed Integer Programming Model} \label{sec:MILP}
	
For practical solution of the optimization problems $\bar{P}_k$, we wish to formulate them as MILP models.  We drop the resource constraints $u\in\mathcal{U}$ from problems $\bar{P}_1,\ldots,\bar{P}_n$ to obtain $P'_1,\ldots,P'_n$, because we wish to analyze the MILP formulations of the SWFs without the complicating factor of resource constraints.  These constraints can later be added to the optimization models before they are solved.  In addition, problems $P'_1,\ldots,P'_n$ contain innocuous auxiliary constraints that make the problems MILP representable.

The MILP model for $P'_1$ follows a different pattern than the models for $P'_2,\ldots,P'_n$, and we therefore treat the two cases separately.  Problem $P'_1$ can be written
\begin{equation}
\begin{array}{ll}
\max \; z_1 \vs \\
{\ds
	z_1 \leq n u_{\langle 1\rangle} + (n-1)\Delta
	+ \hspace{-0.5ex} \sum_{i=2}^n (u_{\langle i\rangle} - u_{\langle 1\rangle} - \Delta)^+ 
} & (a) \vs \\
u_i \geq 0, \;\mbox{all}\;i & (b) \vs \\
u_i - u_j \leq M, \;\text{all } i,j & (c)
\end{array} \label{eq:seq2-1}
\end{equation}
Constraints (c) ensure MILP representability, as explained in \citeauthor{HooWil12} (\citeyear{HooWil12}), because they imply that the hypograph of (\ref{eq:seq2-1}) is a finite union of polyhedra having the same recession cone (\citeauthor{Jeroslow1987} \citeyear{Jeroslow1987}).  Figure~\ref{fig:2personM} illustrates their effect on the feasible region for $n=2$.  The constraints have no practical import for sufficiently large $M$, although for theoretical purposes we assume only $M>\Delta$.  

The MILP model for $P'_1$ can be written as follows:
\begin{equation}
\begin{array}{ll}
\max \; z_1 \vs \\
{\ds
	z_1 \leq (n-1)\Delta + \sum_{i=1}^n v_i 
} & (a) \vs \\
u_i - \Delta \leq v_i \leq u_i - \Delta\delta_i, \;\; i=1,\ldots,n & (b) \vs \\
w \leq v_i \leq w + (M-\Delta)\delta_i, \;\; i=1, \ldots, n & (c) \vs \\
u_i\geq 0, \;\; \delta_i \in \{0,1\}, \; i=1, \ldots, n
\end{array} \label{eq:MILPmodel1}
\end{equation}
The following is proved in \citeauthor{HooWil12} (\citeyear{HooWil12}).
\begin{theorem}
	Model (\ref{eq:MILPmodel1}) is a correct formulation of $P'_1$.
\end{theorem}
	
When $k\geq 2$, the expression (\ref{eq:swfSimple}) for $\bar{F}_k(\vu)$ implies that problem $P'_k$ can be written
\begin{equation}
\begin{array}{ll}
\max \; z_k \vs \\
{\ds
	z_k \leq (n-k+1) \min\{\bar{u}_{i_1} + \Delta, u_{\langle k\rangle} \}
	+ \sum_{i\in I_k} (u_i-\bar{u}_{i_1}-\Delta)^+
} & (a) \vs \\
u_i \geq \bar{u}_{i_{k-1}}, \;\; i\in I_k & (b) \vs \\
u_i-\bar{u}_{i_1} \leq M, \;\; i\in I_k & (c) 
\end{array} \label{eq:seq2}
\end{equation}
The constraints (\ref{eq:seq2}c) are included to ensure that the problem is MILP representable.  Since $\bar{u}_{i_1}$ is a constant, the hypograph is a union of bounded polyhedra whose recession cones consist of the origin only and are therefore identical.  

\begin{figure}
	\centering
	\includegraphics[trim=120 480 230 120, clip, scale=0.9]{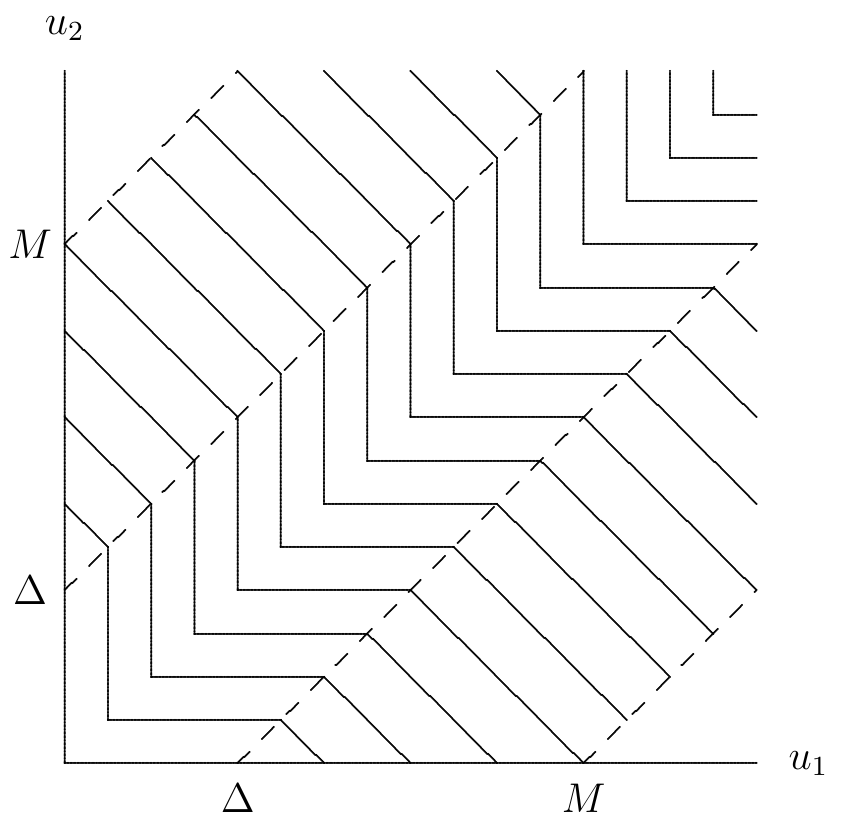}
	\caption{Social welfare function for $F_1(\vu)$ with MILP representable hypograph.}
	\label{fig:2personM}
\end{figure}

The MILP model for $P'_k$ when $k=2,\ldots,n$ is
	
	
\begin{equation}
\begin{array}{ll}
\max \; z_k \vs \\
{\ds
	z_k \leq (n-k+1)\sigma + \sum_{i\in I_k} v_i 
} & (a) \vs \\
0 \leq v_i \leq M\delta_i, \;\; i\in I_k & (b) \vs \\
v_i \leq u_i - \bar{u}_{i_1} - \Delta + M(1-\delta_i), \;\; i\in I_k & (c) \vs \\
\sigma \leq \bar{u}_{i_1} + \Delta & (d) \vs \\
\sigma \leq w  & (e) \vs \\
w \leq u_i, \; i\in I_k & (f) \vs \\
u_i \leq w + M(1-\epsilon_i), \; i\in I_k & (g) \vs \\
{\ds
	\sum_{i\in I_k} \epsilon_i = 1 
} & (h) \vs \\
w \geq \bar{u}_{i_{k-1}} & (i) \vs \\
u_i - \bar{u}_{i_1} \leq M, \; i\in I_k & (j) \vs \\
\delta_i, \epsilon_i \in \{0,1\}, \; i\in I_k
\end{array} \label{eq:MILPmodel}
\end{equation}
	
\begin{theorem} \label{th:correct}
	Model (\ref{eq:MILPmodel}) is a correct formulation of $P'_k$ for $k=2, \ldots, n$.
\end{theorem}
	
\begin{proof}
We first show that given any $(\vu,z_k)$ that is feasible for (\ref{eq:seq2}), where $u_{i_j}=\bar{u}_{i_j}$ for $j=1,\ldots, k-1$, there exist $\bm{v}, \bm{\delta},\bm{\epsilon},w,\sigma$ for which $(\vu,z_k,\bm{v},\bm{\delta},\bm{\epsilon},w,\sigma)$ is feasible for (\ref{eq:MILPmodel}).  Constraint (\ref{eq:MILPmodel}$j$) follows directly from (\ref{eq:seq2}$c$).  To satisfy the remaining constraints in (\ref{eq:MILPmodel}), we set 
\begin{equation}
\begin{array}{l}
{\ds
(\delta_i, \epsilon_i, v_i) = 
\left\{
\begin{array}{ll}
(0,0,0),          & \text{if } u_i - \bar{u}_{i_1} \leq \Delta \mbox{ and } i\neq\kappa \\ 
(0,1,0),   & \text{if } u_i - \bar{u}_{i_1} \leq \Delta \mbox{ and } i=\kappa  \\
(1,0,u_i-\bar{u}_{i_1}-\Delta), & \text{if } u_i - \bar{u}_{i_1} > \Delta \mbox{ and } i\neq\kappa  \\
(1,1,u_i-\bar{u}_{i_1}-\Delta),  & \text{if } u_i - \bar{u}_{i_1} > \Delta \mbox{ and } i=\kappa  
\end{array}
\right\}, \;i \in I_k
} \vs \\
w = u_{\kappa} \vs \\
\sigma = \min\{\bar{u}_{i_1}+\Delta, u_{\kappa}\}
\end{array} \label{eq:assign}
\end{equation}
where $\kappa$ is an arbitrarily chosen index in $I_k$ such that $u_{\kappa}=u_{\langle k\rangle}$.
It is easily checked that these assignments satisfy constraints ($b$)--($h$).  They satisfy ($i$) because (\ref{eq:seq2}$b$) implies that $u_{\kappa}\geq \bar{u}_{i_{k-1}}$.  To show they satisfy (\ref{eq:MILPmodel}$a$), we note that (\ref{eq:MILPmodel}$a$) is implied by (\ref{eq:seq2}$a$) because  $\min\{\bar{u}_{i_1}+\Delta,u_{\kappa}\}\leq \sigma$ and $(u_i-\bar{u}_{i_1}-\Delta)^+\leq v_i$ for $i\in I_k$.  Since (\ref{eq:seq2}$a$) is satisfied by $(\vu,z)$, it follows that (\ref{eq:MILPmodel}$a$) is satisfied by (\ref{eq:assign}).
		
For the converse, we show that for any $(\vu,z_k,\bm{v},\bm{\delta},\bm{\epsilon},w,\sigma)$ that satisfies (\ref{eq:MILPmodel}), $(\vu,z_k)$ satisfies (\ref{eq:seq2}). Constraint (\ref{eq:seq2}$b$) follows from (\ref{eq:MILPmodel}$f$) and (\ref{eq:MILPmodel}$i$), and (\ref{eq:seq2}$c$) is identical to (\ref{eq:MILPmodel}$j$).  To verify that (\ref{eq:seq2}$a$) is satisfied, we let $\kappa$ be the index for which $\epsilon_{\kappa}=1$, which is unique due to (\ref{eq:MILPmodel}$g$).  It suffices to show that (\ref{eq:MILPmodel}$a$) implies (\ref{eq:seq2}$a$) when the remaining constraints of (\ref{eq:MILPmodel}) are satisfied.  For this it suffices to show that 
\begin{equation}
\sigma \leq \min\{\bar{u}_{i_1}+\Delta, u_{\kappa}\}
\label{eq:proof1}
\end{equation}
\begin{equation}
v_i \leq (u_i - \bar{u}_{i_1} - \Delta)^+, \; i\in I_k
\label{eq:proof2}
\end{equation}
(\ref{eq:proof1}) follows from ($d$), ($e$), and ($f$) of (\ref{eq:MILPmodel}).  (\ref{eq:proof2}) follows from ($b$) and ($c$) of (\ref{eq:MILPmodel}).  This proves the theorem.
\end{proof}

\section{Valid Inequalities} \label{sec:sharpness}
	
In this section, we identify some valid inequalities that can strengthen the MILP model of $P'_k$ for $k\geq 2$.  The MILP model (\ref{eq:MILPmodel1}) for $P'_1$ is already sharp, meaning that the inequality constraints of the model describe the convex hull of the feasible set, and there is therefore no benefit in adding valid inequalities.  The sharpness property may be lost when budget constraints are added, but the resulting model may remain a relatively tight formulation.  When $n\leq 3$, the models $P'_k$ for $k\geq 2$ become sharp when the valid inequalities described below are added.  This is not true when $n\geq 4$, but the valid inequalities nonetheless tighten the formulation.  
	
The sharpness of the MILP model (\ref{eq:MILPmodel1}) for $P'_1$ is proved in \citeauthor{HooWil12} (\citeyear{HooWil12}).  We present a simpler proof in Appendix~2.\footnote{The proof in \citeauthor{HooWil12} (\citeyear{HooWil12}) can be simplified by using only the multipliers $\alpha_i = \frac{M}{n\Delta}\left(a_i-1+\frac{\Delta}{M}\right)$ for $i =1, \ldots, n$, because each $a_i\geq 1-\Delta/M$.  The multipliers $\beta_{ij}$ in their proof are unnecessary.}
	
\begin{theorem} \label{th:sharp1}
	The MILP model (\ref{eq:MILPmodel1}) is a sharp representation of $P'_1$ (\ref{eq:seq2-1}).
\end{theorem}

We now describe a class of valid inequalities that can be added to the MILP model (\ref{eq:MILPmodel}) of $P'_k$ for $k\geq 2$ to tighten the formulation.   

\begin{theorem} \label{th:valid}
	The following inequalities are valid for $P'_k$ for $k\geq 2$:
	\begin{align}
	& z_k \leq \sum_{i\in I_k} u_i \label{eq:valid0} \\
	& z_k \leq (n-k+1)u_i 
	+ \beta \hspace{-1.5ex} \sum_{j\in I_k\setminus\{i\}} \hspace{-1.5ex} (u_j - \bar{u}_{i_{k-1}}), \;\; i\in I_k \label{eq:valid}
	\end{align}
	where
	\[
	\beta = \frac{M-\Delta}{M-(\bar{u}_{i_{k-1}}-\bar{u}_{i_1})} = \Big(1-\frac{\Delta}{M}\Big)\Big(1-\frac{\bar{u}_{i_{k-1}}-\bar{u}_{i_1}}{M}\Big)^{-1}
	\]
\end{theorem}

\begin{proof}

It suffices to show that for any $(\vu,z_k,\bm{v},\bm{\delta},\bm{\epsilon},w)$ that satisfies (\ref{eq:MILPmodel}), where $u_{i_j}=\bar{u}_{i_j}$ for \mbox{$j=1,\ldots,k-1$}, the vector $\vu$ satisfies (\ref{eq:valid0}) and (\ref{eq:valid}).  Since we know from Theorem~\ref{th:correct} that $\vu$ is feasible in (\ref{eq:seq2}), it suffices to show that (\ref{eq:seq2}) implies (\ref{eq:valid0}) and (\ref{eq:valid}).  To derive (\ref{eq:valid0}), we write (\ref{eq:seq2}$a$) as
\begin{equation}
z_k \leq \sum_{i\in I_k} \Big( \min\{\bar{u}_{i_1}+\Delta, u_{\langle k\rangle}\}
+(u_i - \bar{u}_{i_1} - \Delta)^+ \Big)
\end{equation}
For any term $i$ in the summation, we consider two cases.  If $u_i\leq \bar{u}_{i_1}+\Delta$, then $u_{\langle k\rangle} \leq \bar{u}_{i_1}+\Delta$ (because $u_{\langle k\rangle}\leq u_i$), and the term reduces to $u_{\langle k\rangle}$.  If $u_i> \bar{u}_{i_1}+\Delta$, term $i$ becomes
\[
\min\{\bar{u}_{i_1}+\Delta, u_{\langle k\rangle}\}
+(u_i - \bar{u}_{i_1} - \Delta)
= \min\{0, u_{\langle k\rangle} - \bar{u}_{i_1} - \Delta\} + u_i \leq u_i
\]
In either case, term $i$ is less than or equal to $u_i$, and (\ref{eq:valid0}) follows.

To establish (\ref{eq:valid}), it is enough to show that (\ref{eq:valid}) is implied by (\ref{eq:seq2}) for each $i\in I_k$.  We consider the same two cases as before.

Case 1: $u_i-\bar{u}_{i_1}\leq\Delta$, which implies $u_{\langle k\rangle}-\bar{u}_{i_1}\leq\Delta$.  Since $\vu$ satisfies (\ref{eq:seq2}$a$), we have
\begin{equation}
z_k \leq (n-k+1)u_{\langle k\rangle} 
+ \hspace{-2ex} \sum_{\substack{j\in I_k\setminus\{i\}\\u_j-\bar{u}_{i_1}>\Delta}} \hspace{-2ex} (u_j-\bar{u}_{i_1}-\Delta)
\label{eq:proof040}
\end{equation}
It suffices to show that this implies 
\begin{equation}
z_k \leq (n-k+1)u_i +
\beta \Big( \hspace{-2ex} \sum_{\substack{j\in I_k\setminus\{i\}\\u_j-\bar{u}_{i_1}\leq\Delta}} \hspace{-2.5ex} (u_j-\bar{u}_{i_{k-1}})
\hspace{0.5ex} + \hspace{-2ex} \sum_{\substack{j\in I_k\setminus\{i\}\\u_j-\bar{u}_{i_1}>\Delta}} \hspace{-2.5ex} (u_j-\bar{u}_{i_{k-1}}) \Big),
\label{eq:proof041}
\end{equation}
because (\ref{eq:proof041}) is equivalent to the desired inequality (\ref{eq:valid}).  But (\ref{eq:proof040}) implies (\ref{eq:proof041}) because $u_{\langle k\rangle} \leq u_i$ by definition of $u_{\langle k\rangle}$, $u_j-\bar{u}_{i_{k-1}}\geq 0$ for all $j\in I_k$ due to (\ref{eq:seq2}$b$), and it can be shown that
\begin{equation}
\beta (u_j-\bar{u}_{i_{k-1}}) \geq u_j - \bar{u}_{i_1} - \Delta
\label{eq:proof401}
\end{equation}
for any $j\in I_k$.  To show (\ref{eq:proof401}), we note that the definition of $\beta$ implies the following identity:
\[
\bar{u}_{i_1}-\bar{u}_{i_{k-1}} +\Delta = (1-\beta)(M + \bar{u}_{i_1} - \bar{u}_{i_{k-1}}).
\]
Adding $(1-\beta)\bar{u}_{i_{k-1}}$ to both sides, we obtain
\begin{equation}
\bar{u}_{i_{k-1}} - \beta \bar{u}_{i_{k-1}} + \Delta = (1-\beta)(M+\bar{u}_{i_1}) \geq (1-\beta)u_j,
\label{eq:proof402}
\end{equation}
where the inequality holds because $M+\bar{u}_{i_1}\geq u_j$ due to (\ref{eq:seq2}$c$).  We obtain (\ref{eq:proof401}) by rearranging (\ref{eq:proof402}).  

Case 2: $u_i-\bar{u}_{i_1}> \Delta$.  It again suffices to show that (\ref{eq:seq2}) implies (\ref{eq:proof041}).  Due to the case hypothesis, we have from (\ref{eq:seq2}$a$) that 
\[
z_k \leq (n-k+1)\min\{\bar{u}_1+\Delta,u_{\langle k\rangle}\}
+ (u_i-\bar{u}_{i_1}-\Delta) 
+ \hspace{-2.5ex} \sum_{\substack{j\in I_k\setminus\{i\}\\u_j-\bar{u}_{i_1}>\Delta}} \hspace{-2.5ex} (u_j-\bar{u}_{i_1}-\Delta)^+
\]
This can be written
\[
z_k \leq (n-k+1)u_i 
- (n-k+1)\Big( u_i - \min\{\bar{u}_1+\Delta,u_{\langle k\rangle}\} \Big)
+ (u_i-\bar{u}_{i_1}-\Delta) 
+ \hspace{-2.5ex} \sum_{\substack{j\in I_k\setminus\{i\}\\u_j-\bar{u}_{i_1}>\Delta}} \hspace{-2.5ex} (u_j-\bar{u}_{i_1}-\Delta)^+
\]
which can be written
\begin{equation}
z_k \leq (n-k+1)u_i 
- (n-k)\Big( u_i - \min\{\bar{u}_1+\Delta,u_{\langle k\rangle}\} \Big)
- \Big( \bar{u}_1 + \Delta - \min\{\bar{u}_1+\Delta,u_{\langle k\rangle} \} \Big)
+ \hspace{-2.5ex} \sum_{\substack{j\in I_k\setminus\{i\}\\u_j-\bar{u}_{i_1}>\Delta}} \hspace{-2.5ex} (u_j-\bar{u}_{i_1}-\Delta)^+
\label{eq:proof45}
\end{equation}
The second term is nonpositive because $u_i>\bar{u}_1+\Delta$ by the case hypothesis, and $u_i\geq u_{\langle k\rangle}$.  The third term is clearly nonpositive.  Thus (\ref{eq:proof45}) implies (\ref{eq:proof041}) because $u_j-\bar{u}_{i_{k-1}}\geq 0$ and (\ref{eq:proof401}) holds for $j\in I_k$ as before.
\end{proof}

	
\section{Modeling Groups of Individuals} \label{sec:group}

In many applications, utility is naturally allocated to groups rather than individuals, where individuals within each group receive an equal allocation.  This occurs, for example, in the health care example we discuss in Section~\ref{sec:exp}, in which groups correspond to classes of patients who have the same disease and a similar prognosis.  In other applications, the number of individuals may be too large for practical solution, since problem $P_i$ must be solved for each individual $i$.  In such cases, individuals can typically be grouped into a few classes within which the individual differences are small or irrelevant, thus making the problem tractable and the results easier to digest.
We therefore modify the above SWFs to accommodate groups rather than individuals.  Since the proofs are similar to those for individuals, we relegate the proofs to Appendix~2.

We suppose there are $n$ groups of possibly different sizes.  We let $u_i$ denote the utility of each individual in group $i$ and $s_i$ the number of individuals in the group.  We can modify $F_1$ for groups in the same fashion as \citeauthor{HooWil12}, namely by replacing the utility variables within group $i$ with the same variable $u_i$.  If $\vu=(u_1,\ldots,u_n)$, this yields the SWF 
\begin{equation}
    \label{eq:group1}
    G_1(\vu) = \Big(\sum_{i=1}^n s_i - 1\Big)\Delta + \Big(\sum_{i=1}^n s_i\Big) u_{\langle 1 \rangle} + \sum_{i=1}^n s_i (u_i - u_{\langle 1 \rangle}-\Delta)^+
\end{equation}
Hooker and Williams prove the following. 
\begin{theorem} \label{th:MILPGroup1}
The problem $P_1^{'}$, modified for groups, is equivalent to the MILP model
\begin{equation}
    \begin{array}{l}
    \max \; z_1 \vs \\
    {\ds
    	z_1 \leq \Big(\sum_{i=1}^n s_i - 1\Big)\Delta + \sum_{i=1}^n s_iv_i 
    } \vs \\
    u_i - \Delta \leq v_i \leq u_i - \Delta\delta_i, \;\; i=1,\ldots,n \vs \\
    w \leq v_i \leq w + (M-\Delta)\delta_i, \;\; i=1, \ldots, n \vs \\
    u_i\geq 0, \;\; \epsilon_i, \delta_i \in \{0,1\}, \; i=1, \ldots, n
    \end{array} \label{eq:MILP1group-new}
\end{equation}
\end{theorem}

If we take a similar approach to modifying $\bar{F}_k(\vu)$ for $k\geq 2$ to accommodate groups, we obtain the following (as shown in Appendix~2):
\begin{equation}
    \label{eq:groupk}
    \bar{G}_k(\vu) =
    \Big( \sum_{i=k}^n s_{\langle i\rangle} \Big)
    \min\{u_{\langle 1\rangle} + \Delta, u_{\langle k\rangle}\}
    + \sum_{i=k}^n s_{\langle i\rangle}(u_{\langle i\rangle} - u_{\langle 1\rangle} - \Delta)^+
\end{equation}
As before, utilities $u_{i_1},\ldots, u_{i_{k-1}}$ are fixed by the solutions of $\bar{P}_1,\ldots,\bar{P}_{k-1}$.  Also $u_{\langle k\rangle}$ is the smallest utility among $\{u_i \;|\; i\in I_k\}$, where again $I_k=\{1,\ldots, n\}\setminus \{i_1,\ldots,i_{k-1}\}$.  If there is a tie for the smallest, we resolve the tie arbitrarily by choosing, say, $u_i$ as the smallest, in which case $s_{\langle k\rangle}=s_i$.  It is clear on inspection that $\bar{G}_k(\vu)$ is well defined because its value is not affected by how the tie is resolved.  Contours for $\bar{G}_2(0,u_2,u_3)$ are illustrated in Fig.~\ref{fig:3personGroups}.  Since $u_{\langle 1\rangle},\ldots,u_{\langle k-1 \rangle}$ are fixed in $\bar{P}_k$, we need only show continuity with respect to the remaining utilities.

\begin{figure}[!t]
	\centering
	\includegraphics[trim=120 450 250 120, clip, scale=0.9]{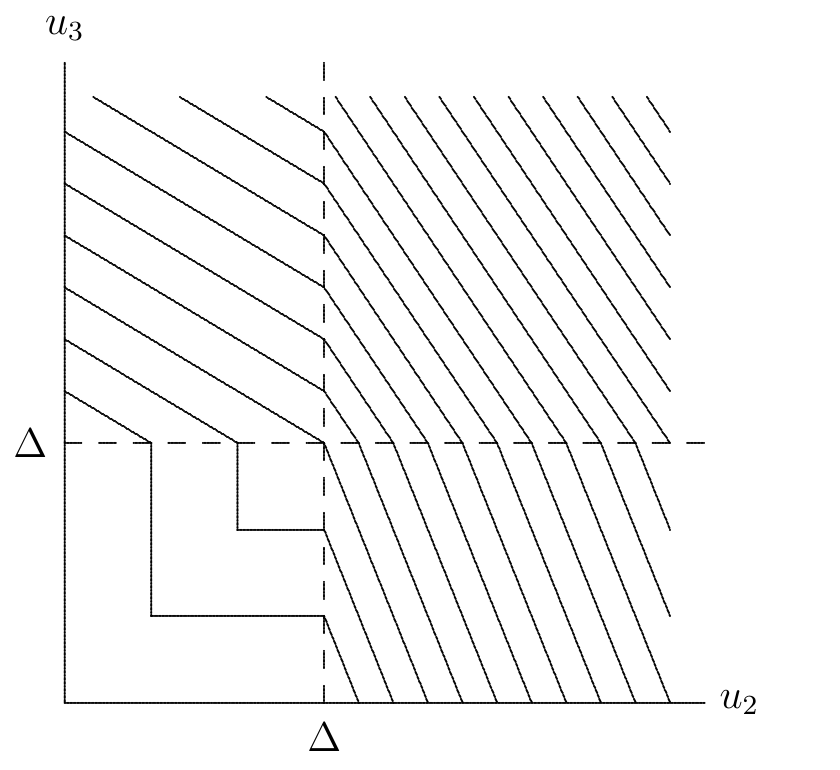}
	\caption{Contours of $\bar{G}_2(0,u_2,u_3)$ with $(s_2,s_3)=(2,3)$, $\Delta = 3$, and contour interval 5.}
	\label{fig:3personGroups}
\end{figure}

\begin{theorem} \label{th:continuity}
	The functions $\bar{G}_k(\vu)$ are continuous in $u_{\langle k \rangle}, \ldots, u_{\langle n \rangle}$ for $k = 1,\ldots,n$.
\end{theorem}

The group version of $P'_k$ becomes 
\begin{equation}
\begin{array}{ll}
\max \; z_k \vs \\
{\ds
	z_k \leq \Big(\sum_{i\in I_k} s_i\Big) \min\{\bar{u}_{i_1} + \Delta, u_{\langle k\rangle} \}
	+ \sum_{i\in I_k} s_i(u_i-\bar{u}_{i_1}-\Delta)^+
} & (a) \vs \\
u_i \geq \bar{u}_{i_{k-1}}, \;\; i\in I_k & (b) \vs \\
u_i-\bar{u}_{i_1} \leq M, \;\; i\in I_k & (c) 
\end{array} \label{eq:seq2-group}
\end{equation}
The MILP model for (\ref{eq:seq2-group}) is very similar to the one we developed for (\ref{eq:seq2}):
\begin{equation}
    \begin{array}{ll}
    \max \; z_k \vs \\
    {\ds
    	z_k \leq \big(\sum_{i\in I_k} s_i\Big) \sigma + \sum_{i\in I_k} s_iv_i 
    } & (a) \vs \\
    (\ref{eq:MILPmodel}b)\mbox{--}(\ref{eq:MILPmodel}j) & (b)\mbox{--}(j) \vs \\
    \delta_i, \epsilon_i \in \{0,1\}, \; i\in I_k
    \end{array}
    \label{eq:MILPkgroup-new}
\end{equation}

\begin{theorem} \label{th:MILPGroup}
The problem $P_k'$, reformulated for groups, is equivalent to (\ref{eq:MILPkgroup-new}) for $k=2,\ldots,n$.
\end{theorem}

\citeauthor{HooWil12} (\citeyear{HooWil12}) prove that (\ref{eq:MILP1group-new}) is a sharp representation of $P_1^{'}$ reformulated for groups.  We present a simpler proof in Appendix~2.

\begin{theorem} \label{th:sharp1-Group}
	The MILP model (\ref{eq:MILP1group-new}) is a sharp representation of $P'_1$ reformulated for groups.
\end{theorem}

Finally, we describe a set of valid inequalities for the MILP model (\ref{eq:MILPkgroup-new}) for $k\geq 2$. 

\begin{theorem} \label{th:valid-group}
	The following inequalities are valid for the group problem $P'_k$ for $k\geq 2$:
	\begin{align}
	& z_k \leq \sum_{i\in I_k} s_iu_i \label{eq:valid0-group} \\
	& z_k \leq \Big(\sum_{j\in I_k} s_i\Big) u_j 
	+ \beta \hspace{-1.5ex} \sum_{j\in I_k\setminus\{i\}} \hspace{-1.5ex} s_j(u_j - \bar{u}_{i_{k-1}}), \;\; i\in I_k \label{eq:valid-group}
	\end{align}
	where
	\[
	\beta = \frac{M-\Delta}{M-(\bar{u}_{i_{k-1}}-\bar{u}_{i_1})} = \Big(1-\frac{\Delta}{M}\Big)\Big(1-\frac{\bar{u}_{i_{k-1}}-\bar{u}_{i_1}}{M}\Big)^{-1}
	\]
\end{theorem}

\section{Step-by-Step Practical Guide} \label{sec:guide}

We now provide a practical guide for incorporating equity and efficiency within an existing optimization model.  We suppose that the model's constraint set is $\mathcal{C}$.  We will define a series of optimization models $\mathcal{M}_1,\ldots, \mathcal{M}_K$ that must be solved to obtain a solution that balances equity and efficiency.  Each $\mathcal{M}_k$ will contain the constraints in $\mathcal{C}$ as well as some additional constraints and an objective function as described below.  If $\mathcal{C}$ is a linear or mixed-integer constraint set, then each $\mathcal{M}_k$ will be an MILP model.  Nonlinear constraints in $\mathcal{C}$ will, of course, result in nonlinear models $\mathcal{M}_k$.

\begin{description}
 \item {\em Step 1.}  Decide what form of utility is to be distributed to the relevant parties.  Utility can be profit, negative cost, or any other benefit that is appropriate to the situation.  In the healthcare resource allocation problem of the next section, utility is the number of quality-adjusted life years experienced by a patient. In the emergency shelter location problem that follows it, utility is negative cost, where cost is measured as the distance one must travel to reach an assigned shelter.  
 
 \item {\em Step 2.}  Identify the parties to whom utility is to be distributed, and for whom fairness is a concern.  The parties could be individual persons or groups of people.  Examples of groups might be organizations, demographic groups, or geographic regions.  Both examples of the next section distribute utility to groups.  In the healthcare case, treatments are allocated to classes of patients who have a similar health condition, and in the shelter allocation case, shelter access is distributed to neighborhoods.  It is assumed that each individual in a given group receives roughly the same utility allotment, so that fairness to the group is interpreted as fairness to the individuals in the group.  When there are a large number of individuals concerned, they can normally be grouped into a few categories of individuals who are similar in relevant respects.  For example, if age-related fairness is desired, the population can be grouped into age categories.  This not only makes the solution of the problem intelligible and relevant to the fairness question, but it eases the computational burden.  The number of optimization models $\mathcal{M}_k$ to be solved is bounded above by the number of parties.
 
 \item {\em Step 3.} Let variable $u_i$ represent the utility allocated to party $i$, or when the party is a group, to each individual in the group $i$.  There are $n$ parties, and $s_i$ is the number of people in group $i$.  Add to $\mathcal{C}$ constraints that define $u_i$ in terms of existing variables, and let $\mathcal{C}'$ be the resulting constraint set.  In some applications, one may wish to model utility as a concave function of the resource that is distributed, to reflect decreasing returns to scale.  For example, the marginal utility of income may decrease with the recipient's wealth.  In such cases, constraints can be introduced to define $u_i$ in terms of resource variables in a nonlinear or piecewise linear fashion.  
 
 \item {\em Step 4.} Let $\mathcal{M}_1$ be (\ref{eq:MILPmodel1}) augmented by the constraints in $\mathcal{C}'$, or if the parties are groups, let $\mathcal{M}_1$ be (\ref{eq:MILP1group-new}) augmented by the constraints in $\mathcal{C}'$.  For $k\geq 2$, let $\mathcal{M}_k$ be (\ref{eq:MILPmodel}) augmented by $\mathcal{C}'$, or in the group model, let $\mathcal{M}_k$ be (\ref{eq:MILPkgroup-new}) augmented by $\mathcal{C}'$.  If desired, each model $\mathcal{M}_k$ for $k\geq 2$ can be tightened by adding the valid constraints (\ref{eq:valid0})--(\ref{eq:valid}) when the parties are individuals, or (\ref{eq:valid0-group})--(\ref{eq:valid-group}) when they are groups.  This may accelerate solution.  Since each $\mathcal{M}_k$ may have multiple optimal solutions, we suggest adding $\epsilon \sum_i u_i$ (or $\epsilon \sum_i s_iu_i$ when the parties are groups) to the objective function as a tie-breaking term, where $\epsilon>0$ is a very small number.
 
 \item {\em Step 5.}  Select $\Delta$ to reflect the desired trade-off between equity and efficiency, where $\Delta=0$ leads to a purely utilitarian solution, and sufficiently large $\Delta$ leads to a purely leximax solution.   Intermediate values of $\Delta$ indicate that parties whose utility is within $\Delta$ of the smallest utility are given a certain amount of priority.  A more detailed interpretation of $\Delta$ is given in Section~\ref{sec:defSWF}.  
 
 \item {\em Step 6.} The optimal solution of $\mathcal{M}_k$ determines the value of the $k$th smallest utility in the socially optimal distribution $\bar{\vu}$.  First, solve $\mathcal{M}_1$, in which $M$ is an upper bound on $u_i-u_j$ for all $i,j$.  Let $\vu^{[1]}$ be the value of $\vu$ in an optimal solution. Let $u_{i_1}$ be a variable for which $u_{i_1}^{[1]}=\min_i\{u_i^{[1]} \;|\; i=1,\ldots,n\}$, and let $\bar{u}_{i_1}=u_{i_1}^{[1]}$.  Now do the following for $k=2,3,\ldots,n$ until $\bar{u}_{i_k} > \bar{u}_{i_1} + \Delta$.  Solve $\mathcal{M}_k$, in which $I_k=\{1,\ldots,n\}\setminus\{i_1,\ldots,i_{k-1}\}$, and $M$ is an upper bound on $u_i-\bar{u}_{i_1}$ for all $i\in I_k$. Let $\vu^{[k]}$ be an optimal solution.  Let $u_{i_k}$ be a variable for which $u_{i_k}^{[k]}=\min_i\{u_i^{[k]} \;|\; i\in I_k\}$, and let $\bar{u}_{i_k}=u_{i_k}^{[k]}$.  If $\bar{u}_{i_k}>\bar{u}_{i_1}+\Delta$ and $k<n$, let $u_i=u_i^{[k]}$ for $i\in I_k\setminus\{k\}$.
 
 \item {\em Step 7.}  The resulting utility vector $\bar{\vu}$ is an optimal solution of the social welfare problem, using parameter $\Delta$.  At this point, decision makers may wish to re-solve the problem for different values of $\Delta$ to obtain a solution with the desired balance of equity and utility.  In particular, they can select a parameter value for which utilities $\bar{u}_i$ within $\Delta$ of $\bar{u}_{i_1}$ represent parties that should receive priority. 

\end{description}

\section{Applications} \label{sec:exp}
We now implement our approach on a healthcare resource allocation problem and a disaster management problem.  We solve all MILP instances using Gurobi 8.1.1 on a desktop PC running Windows 10.

\subsection{Healthcare Resource Allocation} \label{sec:healthexp}

A proper balance between fairness and efficiency is crucial in the allocation of healthcare resources.
\citeauthor{HooWil12} (\citeyear{HooWil12}) study a problem in which treatments are made available to patients on the basis of their disease and prognosis.  In discussing this case, we caution that the results we report should not be taken as general recommendations for the allocation of medical resources.  They are based on cost and clinical data that are outdated and specific to a particular set of circumstances.  We nonetheless use this example because it allows comparison with the published Hooker-Williams (H-W) results on the same problem instance.

Patients are divided into groups based on their disease and prognosis.  There is one treatment potentially available to each patient group, and for policy consistency, it is provided to either all or none of the group members.  Binary variable $y_i$ is 1 if group $i$ receives the recommended treatment and 0 otherwise.  
The average utility $u_i$ experienced by members of group $i$ is measured in terms of quality adjusted life years (QALYs); $q_i$ is the net gain in QALYs for a member of group $i$ when receiving the recommended treatment, and $\alpha_i$ is the expected QALYs experienced with medical management without the treatment. Thus
\begin{equation}
u_i = \alpha_i + q_i y_i, \;\; i=1,\ldots, n
\label{eq:health1}
\end{equation}
The budget constraint is 
\begin{equation}
\sum_{i}^n s_i c_i y_i \leq B
\label{eq:health2}
\end{equation}
where $s_i$ is the group size, $c_i$ the cost of treating one patient in group $i$, and $B$ the total available budget.  The specific data can be found in Table~1 of the H-W paper. The budget is set so as to force some hard decisions.  Referring to the notation of the previous section, the original problem constraint set $\mathcal{C}$ consists of (\ref{eq:health2}) and $y_i\in \{0,1\}$ for $i=1, \ldots, n$, and $\mathcal{C}'$ is obtained by adding constraints (\ref{eq:health1}).

The H-W results are reproduced here in Table~\ref{table:result-HW}, in which the columns indicate solutions values of $y_i$ for the 33 patient groups and various ranges of $\Delta$.   The treatments are pacemaker implant, hip replacement, aortic valve replacement, coronary artery bypass grafting (CABG), heart and kidney transplant, and kidney dialysis.  Three types pf CABG surgery are distinguished (left main, double, and triple bypass), and kidney dialysis patients are distinguished by years of life expectancy with dialysis.  Most of these categories are further divided into one, two, or three patient groups representing the degree of severity of the disease.  The last column indicates the average number of QALYs per patient for each $\Delta$ range.

		\begin{table}
		\centering
		\caption{Results of healthcare example from Table~2 of \citeauthor{HooWil12} (\citeyear{HooWil12}).  A $1$ indicates that the treatment is given to all members of a patient group, and a $0$ indicates that the treatment is given to none.} \label{table:result-HW}
		\vs
		{\footnotesize
		\begin{tabular}{ccccccccccccccc}
		$\Delta$ & Pace- & Hip & Aortic & \multicolumn{3}{c}{CABG} & Heart & Kidney & \multicolumn{5}{c}{Kidney dialysis} & Avg. \\
		range & maker & repl. & valve & L & 3 & 2 & trans. & trans. & $<$1 & 1-2 & 2-5 & 5-10 & $>$10 & QALYs. \\  \hline 
		0-3.3  & 111 & 111 & 111 & 111 & 111 & 111 & 1 & 11 & 0 & 0 & 000 & 0000 & 000 & 7.54 \\
		3.4-4.0 & 111 & 111 & 111 & 111 & 111 & 111 & 0 & 11 & 1 & 0 & 000 & 0000 & 000 & 7.54 \\
		4.0-4.4 & 111 & 111 & 111 & 111 & 111 & 111 & 0 & 01 & 1 & 0 & 000 & 0000 & 001 & 7.51 \\
		4.5-5.01 & 111 & 011 & 111 & 111 & 111 & 111 & 1 & 01 & 1 & 0 & 000 & 0000 & 011 & 7.43 \\
		5.02-5.55 & 111 & 011 & 011 & 111 & 111 & 111 & 1 & 01 & 1 & 0 & 000 & 0001 & 011 & 7.36 \\
		5.56-5.58 & 111 & 011 & 011 & 111 & 111 & 011 & 0 & 01 & 1 & 0 & 000 & 0001 & 111 & 7.36 \\
		5.59 & 111 & 011 & 011 & 110 & 111 & 111 & 0 & 01 & 1 & 0 & 000 & 0001 & 111 & 7.20 \\
		5.60-13.1 & 111 & 111 & 111 & 101 & 000 & 000 & 1 & 11 & 1 & 0 & 111 & 1111 & 111 & 7.06 \\
		13.2-14.2 & 111 & 011 & 111 & 011 & 000 & 000 & 1 & 11 & 1 & 1 & 111 & 1111 & 111 & 7.03 \\
		14.3-15.4 & 111 & 111 & 111 & 011 & 000 & 000 & 1 & 11 & 1 & 1 & 101 & 1111 & 111 & 7.13 \\
		15.5 up & 111 & 011 & 111 & 011 & 001 & 000 & 1 & 11 & 1 & 0 & 011 & 1111 & 111 & 7.19 \\
		 \hline
		\end{tabular}
		}
	\end{table}

The results contain several interesting features, but most obvious is the transfer of resources from heart bypass surgery to dialysis as $\Delta$ increases.  Kidney dialysis is quite costly, because the treatment is ongoing rather than a one-time event such as surgery.  The payoff in QALYs per unit cost is therefore relatively low, and bypass surgery is selected when the utilitarian objective dominates (smaller values of $\Delta$).  As $\Delta$ increases, resources are transferred to dialysis patients, who are the worst off without treatment; heart bypass patients tend to have a fairly long life expectancy without the surgery.  

Most relevant here is the fact that the average QALYs per patient decrease relatively little as $\Delta$ increases, as can be seen in the last column of Table~\ref{table:result-h}.  This is due to the fact that the H-W method does not take into account the utility levels of patients in the fair region (i.e., within $\Delta$ of the lowest), except for the very lowest.  This results in a large space of alternate optimal solutions, many of them quite different from each other.  To deal with this indeterminacy, the H-W experiments break ties by adding $\epsilon\sum_i s_iu_i$ to the objective function.  This means that utilities in the fair region (except the lowest) are treated in a  utilitarian fashion.  Thus as $\Delta$ increases, the solution becomes basically utilitarian again, except that the welfare of the very worst-off patient is maximized.    

Our results appear in Tables~\ref{table:result-h} and~\ref{table:result-h-bt}, one with and one without tie-breaking.  We discuss the merits of tie-breaking below, but in our experiments it has a rather small effect on the solution and does not change the analysis to follow.  The computation time for a given $\Delta$ is negligible, almost always less than 0.5 second, even though there are 33 groups.   The average utility per patient is plotted in Fig.~\ref{fig:health}.  The fluctuations in utility as $\Delta$ increases indicate the subtlety and complexity of balancing leximax fairness and utilitarianism.  The discontinuities are due to the discreteness of the problem.  When a group moves into or out of the fair region, the impact on average utility can be significant even when the effect on social welfare is small.  

	\begin{table}[!b]
		\centering
		\caption{Results of healthcare example, without tie-breaking.   
		}
		\label{table:result-h}
		\vs
		{\footnotesize
		\begin{tabular}{ccccccccccccccc}
		$\Delta$ & Pace- & Hip & Aortic & \multicolumn{3}{c}{CABG} & Heart & Kidney & \multicolumn{5}{c}{Kidney dialysis} & Avg. \\
		range & maker & repl. & valve & L & 3 & 2 & trans. & trans. & $<$1 & 1-2 & 2-5 & 5-10 & $>$10 & QALYs. \\ \hline 
		0-0.2  & 111 & 111 & 111 & 111 & 111 & 111 & 1 & 11 & 0 & 0 & 000 & 0000 & 000 & 7.544 \\
		0.3  & 111 & 111 & 111 & 111 & 111 & 011 & 1 & 11 & 0 & 0 & 100 & 0000 & 001 & 7.542 \\
		0.4 & 111 & 111 & 111 & 111 & 011 & 011 & 1 & 11 & 0 & 1 & 100 & 1000 & 000 & 7.516 \\
		0.5 & 111 & 111 & 111 & 111 & 011 & 011 & 0 & 11 & 0 & 1 & 110 & 1000 & 000 & 7.514 \\
		0.6 & 111 & 111 & 111 & 111 & 011 & 001 & 1 & 01 & 0 & 1 & 111 & 1100 & 000 & 7.453 \\
		0.7 & 111 & 111 & 111 & 111 & 011 & 001 & 0 & 00 & 0 & 1 & 111 & 1100 & 100 & 7.405 \\
		0.8 & 111 & 111 & 111 & 011 & 011 & 001 & 1 & 00 & 0 & 1 & 111 & 1100 & 110 & 7.354 \\
		0.9 & 111 & 111 & 111 & 011 & 001 & 001 & 1 & 00 & 0 & 1 & 111 & 1101 & 110 & 7.283 \\
		1-1.9 & 111 & 111 & 011 & 011 & 001 & 001 & 1 & 00 & 0 & 1 & 111 & 1101 & 111 & 7.226 \\
		2-3.3 & 111 & 111 & 111 & 011 & 001 & 000 & 1 & 00 & 0 & 1 & 111 & 1111 & 111 & 7.206 \\
		3.4-4.55 & 011 & 111 & 111 & 011 & 001 & 000 & 1 & 00 & 1 & 1 & 111 & 1111 & 111 & 7.091 \\
		4.56-5.06 & 001 & 111 & 111 & 011 & 001 & 001 & 1 & 00 & 1 & 1 & 111 & 1111 & 111 & 6.940 \\
	    5.07-5.34 & 011 & 111 & 111 & 011 & 001 & 000 & 1 & 00 & 1 & 1 & 111 & 1111 & 111 & 7.091 \\
	    5.35-6.59 & 111 & 111 & 111 & 001 & 001 & 000 & 1 & 01 & 1 & 1 & 111 & 1111 & 111 & 7.123 \\
	    6.60-8.43 & 011 & 111 & 111 & 011 & 001 & 000 & 1 & 00 & 1 & 1 & 111 & 1111 & 111 & 7.091 \\
	    8.44-11.5 & 001 & 111 & 111 & 011 & 001 & 001 & 1 & 00 & 1 & 1 & 111 & 1111 & 111 & 6.940 \\
	    11.6-13 & 010 & 111 & 111 & 011 & 001 & 001 & 1 & 00 & 1 & 1 & 111 & 1111 & 111 & 6.802 \\
	    13 up & 000 & 110 & 110 & 011 & 001 & 001 & 1 & 00 & 1 & 1 & 111 & 1111 & 111 & 5.942 \\
		 \hline
		\end{tabular}
		}
	\end{table}
	
	\begin{table}[!]
		\centering
		\caption{Results of the healthcare example, with tie-breaking:
		}\label{table:result-h-bt}
		\vs
		{\footnotesize
		\begin{tabular}{ccccccccccccccc}
		$\Delta$ & Pace- & Hip & Aortic & \multicolumn{3}{c}{CABG} & Heart & Kidney & \multicolumn{5}{c}{Kidney dialysis} & Avg. \\
		range & maker & repl. & valve & L & 3 & 2 & trans. & trans. & $<$1 & 1-2 & 2-5 & 5-10 & $>$10 & QALYs. \\ \hline 
		0-0.2  & 111 & 111 & 111 & 111 & 111 & 111 & 1 & 11 & 0 & 0 & 000 & 0000 & 000 & 7.544 \\
		0.3  & 111 & 111 & 111 & 111 & 111 & 011 & 1 & 11 & 0 & 0 & 100 & 0000 & 001 & 7.542 \\
		0.4 & 111 & 111 & 111 & 111 & 011 & 011 & 1 & 11 & 0 & 1 & 100 & 1000 & 000 & 7.516 \\
		0.5 & 111 & 111 & 111 & 111 & 011 & 011 & 0 & 11 & 0 & 1 & 110 & 1000 & 000 & 7.514 \\
		0.6 & 111 & 111 & 111 & 111 & 011 & 001 & 1 & 01 & 0 & 1 & 111 & 1100 & 000 & 7.453 \\
		0.7 & 111 & 111 & 111 & 111 & 011 & 001 & 0 & 00 & 0 & 1 & 111 & 1100 & 100 & 7.405 \\
		0.8 & 111 & 111 & 111 & 011 & 011 & 001 & 1 & 00 & 0 & 1 & 111 & 1100 & 110 & 7.354 \\
		0.9 & 111 & 111 & 111 & 011 & 001 & 001 & 1 & 00 & 0 & 1 & 111 & 1101 & 110 & 7.283 \\
		1-1.9 & 111 & 111 & 011 & 011 & 001 & 001 & 1 & 00 & 0 & 1 & 111 & 1101 & 111 & 7.226 \\
		2 & 111 & 111 & 111 & 011 & 011 & 000 & 1 & 00 & 0 & 1 & 111 & 1111 & 011 & 7.274 \\
		2.1-3.2 & 111 & 111 & 111 & 011 & 001 & 000 & 1 & 00 & 0 & 1 & 111 & 1111 & 111 & 7.206 \\
		3.3 & 111 & 111 & 111 & 011 & 001 & 000 & 1 & 00 & 1 & 1 & 111 & 1101 & 111 & 7.197 \\
		3.4-4.55 & 011 & 111 & 111 & 011 & 001 & 000 & 1 & 00 & 1 & 1 & 111 & 1111 & 111 & 7.091 \\
		4.56-5.06 & 001 & 111 & 111 & 011 & 001 & 001 & 1 & 00 & 1 & 1 & 111 & 1111 & 111 & 6.940 \\
	    5.07-5.34 & 011 & 111 & 111 & 011 & 001 & 000 & 1 & 00 & 1 & 1 & 111 & 1111 & 111 & 7.091 \\
	    5.35-5.59 & 111 & 111 & 111 & 001 & 001 & 000 & 1 & 01 & 1 & 1 & 111 & 1111 & 111 & 7.123 \\
	    5.60-6.50 & 111 & 111 & 111 & 011 & 000 & 000 & 1 & 01 & 1 & 1 & 111 & 1111 & 111 & 7.099 \\
	    6.51-6.59 & 111 & 111 & 111 & 001 & 001 & 000 & 1 & 01 & 1 & 1 & 111 & 1111 & 111 & 7.123 \\
	    6.60-7.59 & 011 & 111 & 111 & 011 & 001 & 000 & 1 & 00 & 1 & 1 & 111 & 1111 & 111 & 7.091 \\
	    7.60-9.6 & 101 & 111 & 111 & 011 & 001 & 000 & 1 & 01 & 1 & 1 & 111 & 1111 & 111 & 7.012 \\
	    9.7-12.2 & 000 & 111 & 111 & 011 & 001 & 001 & 1 & 01 & 1 & 1 & 111 & 1111 & 111 & 6.607 \\
	    12.3-12.7 & 010 & 111 & 111 & 011 & 001 & 001 & 1 & 00 & 1 & 1 & 111 & 1111 & 111 & 6.802 \\
	    12.8-13 & 000 & 111 & 110 & 011 & 001 & 001 & 1 & 01 & 1 & 1 & 111 & 1111 & 111 & 6.309 \\
	    13.1 & 000 & 111 & 111 & 011 & 001 & 001 & 1 & 00 & 1 & 1 & 111 & 1111 & 111 & 6.058 \\
	    13.2-14.5 & 000 & 110 & 110 & 011 & 001 & 001 & 1 & 00 & 1 & 1 & 111 & 1111 & 111 & 5.942 \\
	    14.6-15.5 & 000 & 111 & 110 & 111 & 001 & 001 & 1 & 00 & 1 & 1 & 111 & 1111 & 111 & 6.325 \\
	    15.6 up & 000 & 111 & 110 & 011 & 001 & 001 & 1 & 01 & 1 & 1 & 111 & 1111 & 111 & 6.309 \\
		 \hline
		\end{tabular}
		}
	\end{table}
	\begin{figure}
	    \centering
	    \includegraphics[scale=0.3,trim=130 25 140 60, clip]{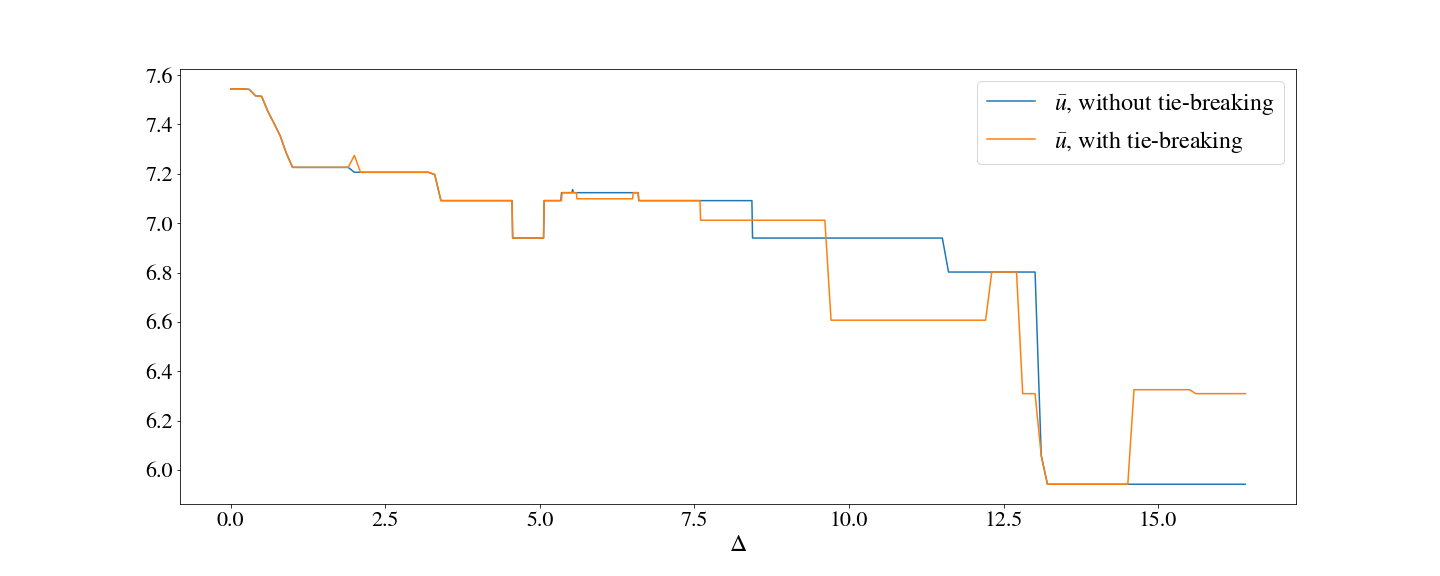}
	    \caption{Average QALYs of optimal outcomes in healthcare example}
	    \label{fig:health}
	\end{figure}

We note first that the average utility per patient drops considerably as $\Delta$ increases, indicating that equity plays a larger role for $\Delta>0$ than in the H-W solution.  Kidney dialysis enters the solution for much smaller values of $\Delta$, and the more seriously ill kidney patients enter first, the reverse of what occurs in the H-W solution.  This reflects the fact that our solution is sensitive to the utility levels of all disadvantaged patient groups rather than only the very worst-off.  The trade-off parameter $\Delta$ also conveys more information than in the H-W case.  Resources begin to transfer to dialysis patients when $\Delta$ is only a fraction of a QALY.  A value of $\Delta$ this small dictates that patients receive priority only when their life expectancy is within a few months of the worst-off.  One can therefore conclude that most dialysis patients are nearly as badly off as the worst-off patient in a socially optimal solution.  The H-W solution does not provide this information, because it brings in most dialysis patients only when $\Delta$ is greater than 5 QALYs or so.

There are other differences with the H-W solution.  Heart bypass surgery remains in the solution for the most seriously ill patients, with some exceptions, through the entire range of $\Delta$.  This is again because the solution is sensitive to their disadvantaged position even though they are not the worst-off.  Pacemakers now drop out of the solution for large $\Delta$, even though pacemaker implantation is relatively inexpensive.  This is because the pacemaker patients are better off without treatment than any of the other patients and therefore cease to receive priority as equity becomes more important.  In general, these results indicate that incorporating leximax rather than maximin fairness in a social welfare function yields more nuanced solutions that are more sensitive to equity considerations.

Finally, we note that due to the nonmonotonicity of average utility with respect to $\Delta$, solutions can sometimes dominate others with respect to both efficiency and equity.  For example, solutions with $\Delta\geq 14.5$ dominate solutions with $13\leq \Delta \leq 14.5$ (see Fig.~\ref{fig:health}), and the latter might be rejected on that basis.


We experimented with tie-breaking because alternate optimal solutions typically exist at each stage of the optimization procedure, if not to the same extent as when maximin fairness is used.  Tables~\ref{table:result-h} and~\ref{table:result-h-bt} indicate that tie breaking has only a minor effect on the solution but results in greater sensitivity to $\Delta$, and it may therefore allow greater control over the equity/efficiency tradeoff.  There is little difference in the average utility for $\Delta < 7.5$, as is evident in Fig.~\ref{fig:health}, while differences can be greater for larger $\Delta$.  This is perhaps because larger values of $\Delta$ lead to more iterations in the solution procedure (i.e., more sequential optimization problems are solved), and there is greater opportunity to introduce alternate optimal solutions with different utility values.  Tie-breaking may be desirable for policy-making purposes due to greater consistency in the outcome, since otherwise the solution obtained is a random choice among the optima.  Tie-breaking can also allow one to obtain greater utility at the same level of $\Delta$---although not always (as when $10\leq \Delta\leq 12$), since breaking ties early in the sequence can fix utilities that preclude solutions with greater utility later in the sequence.  In any case, alternate optimal solutions are a regular feature of combinatorial optimization problems, and we should expect nothing different here.

\subsection{Shelter Location and Assignment} \label{sec:shelterexp}

Disaster preparation and post-disaster response are important elements of humanitarian operations, in which equity is an essential consideration. We apply our approach to the shelter location and assignment problem investigated in \citeauthor{Sibel2019inequity}\ (\citeyear{Sibel2019inequity}). There are two sets of decisions: where to construct shelters in preparation for natural disasters, and how to assign one shelter to each potentially affected residential area. When a disaster strikes, those affected will evacuate to their assigned shelter for protection and supplies.  \citeauthor{Sibel2019inequity}\ solve a model with multiple scenarios representing possible demands and street disruptions, but we simplify the problem by removing the stochastic element so as to clarify the equity/efficiency trade-off. 

Utility is measured as negative cost, where cost is taken to be the travel distance between a residential area and its assigned shelter. A conventional efficiency objective is to minimize the average travel distance among all individuals. Optimizing the efficiency objective alone may result in unreasonable solutions that force people in some areas to travel a long distance to their shelter. We argue that a just and reasonable solution should try to reduce the distance for less advantaged areas without increasing average travel distance inordinately. We show that our approach is applicable to obtain such a trade-off between equity and efficiency.

To formulate a MILP model of the problem, suppose $m$ is the number of candidate locations for shelter, and $n$ is the number of population areas. For $j \in \{1,\ldots,m\}$, $c_j$ is the shelter capacity, and $e_j$ is the cost of opening a shelter at location $j$. For $i \in \{1,\ldots,n\}$, $s_i$ is the population of area $i$. We suppose that each person living in area $i$ must travel a distance of $D_{ij}$ to reach location $j$. We use two sets of binary decision variables: for location $j$, $y_j = 1$ means that a shelter is open at location $j$; for area $i$, $X_{ij} = 1$ means all persons living in the area are assigned to shelter $j$. Additionally, we impose a budget constraint requiring that the total cost of opening shelters does not exceed $B$. The model of \citeauthor{Sibel2019inequity}\ assumes that each shelter is large enough for the entire population of any individual area, so that it is unnecessary to split areas between shelters. The constraint set $\mathcal{C}$ for the problem is therefore
\begin{align*}
    &\sum_{j=1}^m X_{ij} = 1, \;\; i=1, \ldots, n  \\
    &\sum_{i=1}^n s_i X_{ij} \leq c_j y_j, \;\; j=1,\ldots,m \\ 
    &\sum_{j=1}^m e_j y_j \leq B \\
    &X_{ij}, \; y_j \in \{0,1\}, \;\; i=1,\ldots, n, \; j=1,\ldots, m 
\end{align*}
The utility $u_i$ of each person in area $i$ is defined by the following constraints, which are added to $\mathcal{C}$ to obtain the constraint set $\mathcal{C}'$: 
\[
u_i = -\sum_{j=1}^m D_{ij}X_{ij}, \;\; i=1,\ldots, n
\]

We generate problem instances with one of the methods used by \citeauthor{Sibel2019inequity} Instances of the capacitated warehouse location problem from \citeauthor{beasley1988algorithm} (\citeyear{beasley1988algorithm}) are converted to shelter location instances by identifying shelters with warehouses, residential areas with customers, and population counts $s_i$ with customer demands.  The distance $D_{ij}$ is taken to be $C_{ij}/s_i$, where $C_{ij}$ is the cost of meeting all of customer $i$'s demand from warehouse $j$.  We use instances cp92 (25 locations) and cp122 (50 locations), both having 50 customers.  The budget $B$ is set to 150000 for cp92 and 300000 for cp122.

The resulting average per capita utility is plotted against $\Delta$ in Fig.~\ref{fig:cap}, both with and without tie-breaking.  In this case, tie-breaking has very little effect, except for the pure leximax solution ($\Delta\geq 37$) of cap92.  Computation time is plotted against $\Delta$ in Fig.~\ref{fig:shelter-time}.  Run time is greater for intermediate values of $\Delta$, but never more than 18 seconds, and less than 10 seconds in most cases.

\begin{figure}[!t]
    \centering
    \begin{subfigure}{.5\textwidth}
  \centering
  \includegraphics[scale=0.3,trim=50 25 80 60, clip]{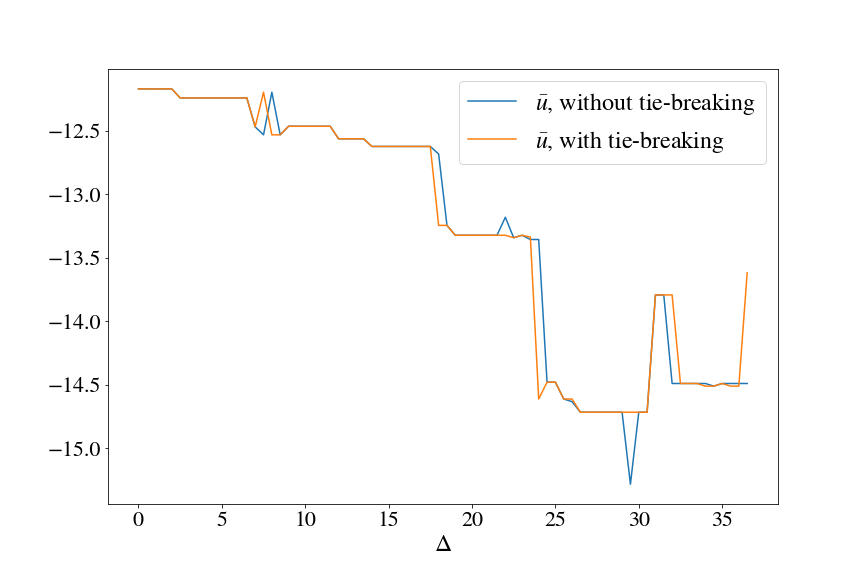}
  \caption{Instance cp92 ($n=50$, $m=25$)}
  \label{fig:cap92}
\end{subfigure}%
\begin{subfigure}{.5\textwidth}
  \centering
  \includegraphics[scale=0.3,trim=50 25 80 60, clip]{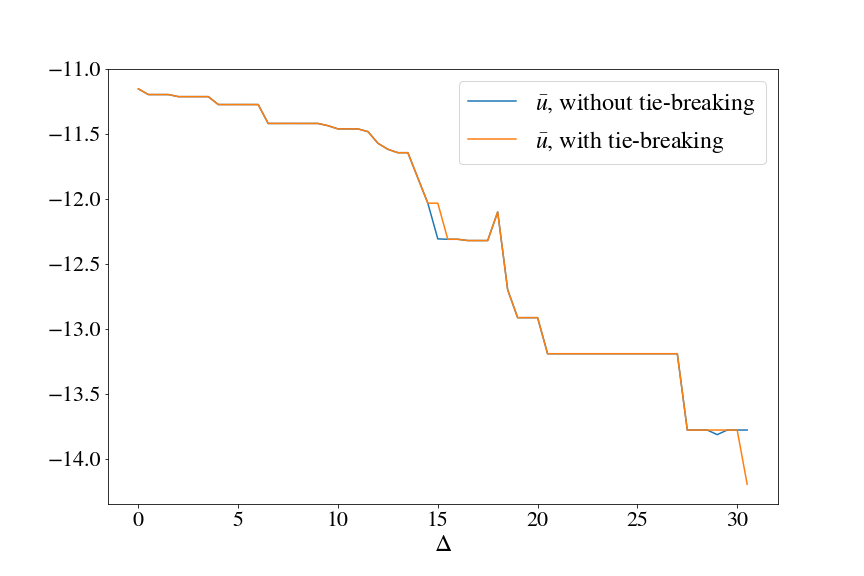}
  \caption{Instance cp122 ($n=50$, $m=50$)}
  \label{fig:cap122}
\end{subfigure}
    \caption{Average utility per person for the shelter location instances}
    \label{fig:cap}
\end{figure}

\begin{figure}[!t]
    \centering
    \begin{subfigure}{.5\textwidth}
  \centering
  \includegraphics[width=\linewidth]{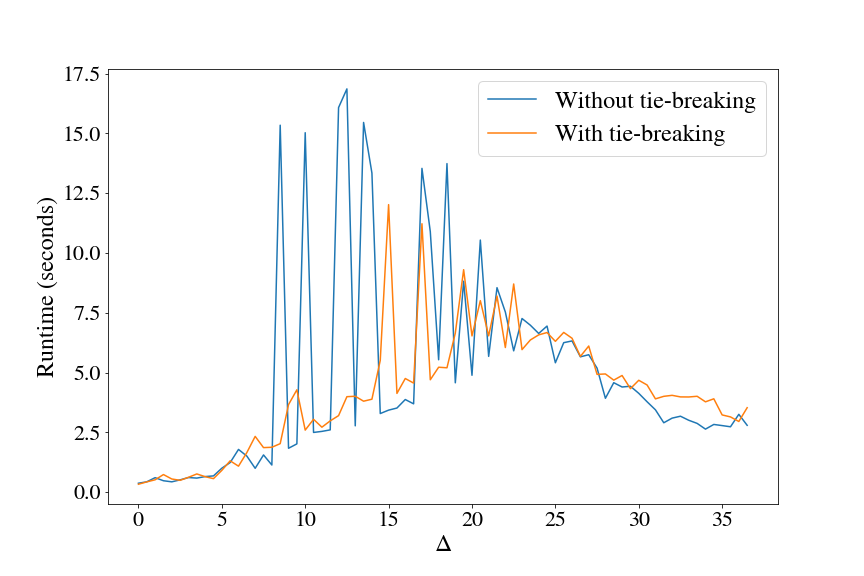}
  \caption{Instance cp92 ($n=50$, $m=25$)}
  \label{fig:cap92-t}
\end{subfigure}%
\begin{subfigure}{.5\textwidth}
  \centering
  \includegraphics[width=\linewidth]{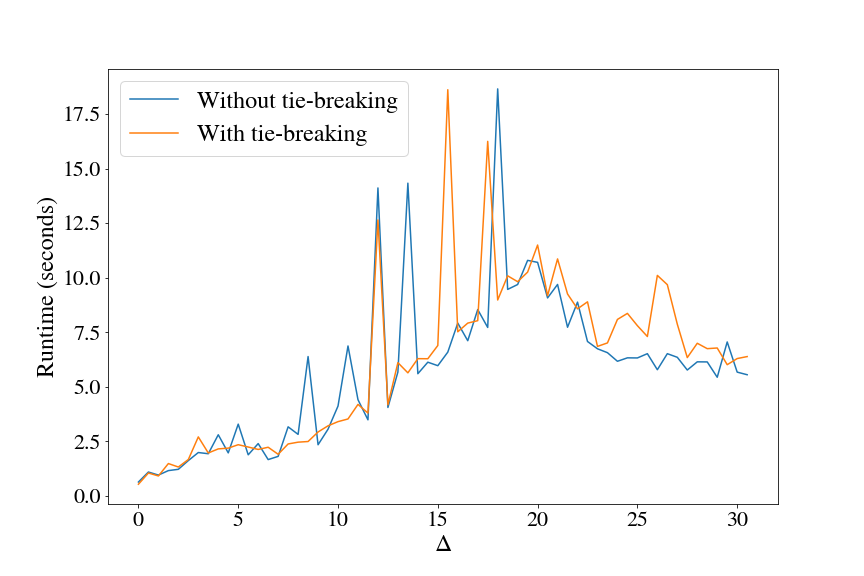}
  \caption{Instance cp122 ($n=50$, $m=50$)}
  \label{fig:cap122-t}
\end{subfigure}
    \caption{Runtime in shelter allocation example}
    \label{fig:shelter-time}
\end{figure}

Figure~\ref{fig:cap} shows that, as expected, there is a generally decreasing trend in average per capita utility as $\Delta$ increases. The trend is almost strictly monotonic in the cap122 instance, but utility jumps up when $\Delta$ is about 30 in the cap92 instance.  The instability is mostly due to the inherent complexity of trade-offs in this instance, and the fact that neighborhood populations vary greatly.  The nonmonotonicity can assist decision making, as solutions corresponding to $\Delta$ values between 23 and 30 are dominated by more equitable solutions with respect to utility and may be discarded on this basis.

\begin{figure}[!t]
    \centering
    \begin{subfigure}{.5\textwidth}
  \centering
  \includegraphics[scale=0.38,trim=50 25 80 60, clip]{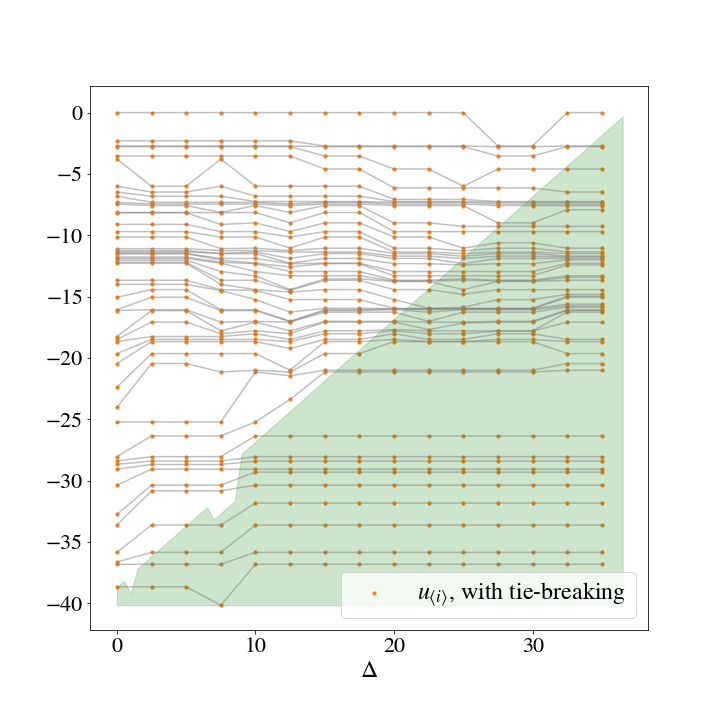}
  \caption{Instance cp92 ($n=50$, $m=25$)}
  \label{fig:cap92}
\end{subfigure}%
\begin{subfigure}{.5\textwidth}
  \centering
  \includegraphics[scale=0.38,trim=50 25 80 60, clip]{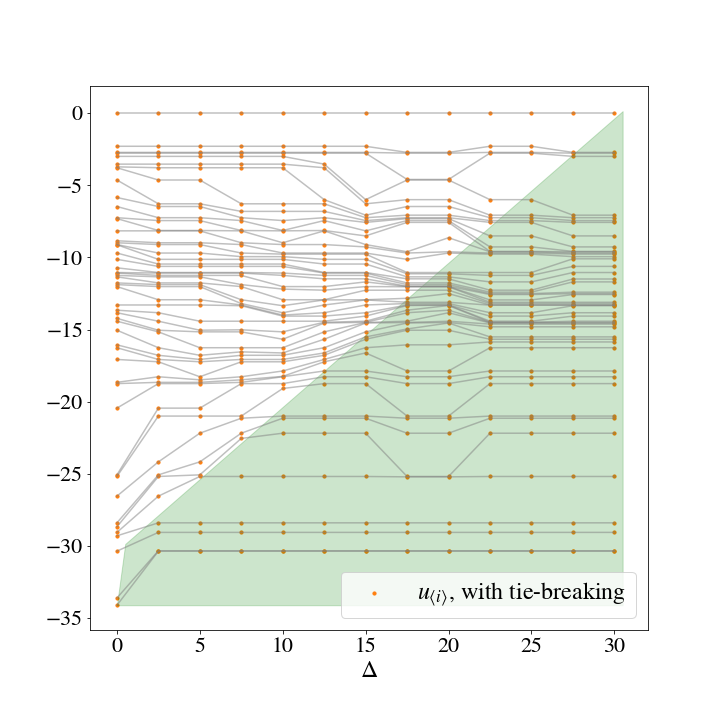}
  \caption{Instance cp122 ($n=50$, $m=50$)}
  \label{fig:cap122}
\end{subfigure}
    \caption{Utility distributions in the shelter allocation instances}
    \label{fig:shelter-u}
\end{figure}

To study the effect of increasing $\Delta$ values more closely, we plot the socially optimal utility distribution (with tie-breaking) against $\Delta$ in Fig.~\ref{fig:shelter-u}. These plots show the evolution of per capita utility in individual neighborhoods as $\Delta$ increases.  The shaded region indicates which utilities are in the fair region (within $\Delta$ of the worst).  We see immediately that the problem is highly constrained, because the lowest utilities quickly reach a plateau and remain at a low level even for large $\Delta$ values.  These neighborhoods are located at a considerable distance from candidate shelter locations, and so they remain disadvantaged even when given high priority.  In this type of situation, it is particularly important to use a leximax rather than a maximin criterion of fairness, so as to take into account the situation of disadvantaged neighborhoods other than the very worst-off.  

The leximax component yields the desired result, because several of the less advantaged neighborhoods improve their status as $\Delta$ increases.  At the same time, some of the more privileged neighborhoods lose utility as they begin to sacrifice somewhat for the sake of the more remote neighborhoods.  Interestingly, the less advantaged neighborhoods typically start to benefit from an increasing $\Delta$ shortly before they enter the fair region.   The reason for this is that when a low utility level enters the fair region, it immediately becomes suboptimal due to the greater weight it receives there.  It is therefore pushed up to one of the discrete feasible levels outside the fair region.  We also see in the plot for cp122 that some of the disadvantaged neighborhoods temporarily sacrifice utility as $\Delta$ enters the range from 15 to 22.  This is to benefit better-off neighborhoods to which priority is extended at this point.  Their utility is restored for larger values of $\Delta$ as the solution becomes still more equity-based and additional utility is transferred from the best-off neighborhoods.


\section{Conclusion} \label{sec:conclusion}
We propose a new systematic approach to balancing efficiency and equity in an optimization model, in which utility serves as the efficiency measure and Rawlsian leximax fairness as the equity measure. We are guided by the intuition that less advantaged utility recipients should be prioritized, but without an excessive reduction in overall utility.

We combine leximax and utilitarian criteria by defining a trade-off parameter $\Delta$ and dividing the feasible utility range into a \emph{fair} region and a \emph{utilitarian} region, where the fair region consists of utilities within $\Delta$ of the utility of the worst-off party.  Leximax fairness is the dominating objective in the fair region, and utility dominates otherwise. Thus a single parameter allows a decision maker to control the balance between equity and efficiency by deciding which parties are sufficiently disadvantaged---that is, sufficiently near the worst-off---to deserve some degree of priority.  

For a given optimization model, we show how to accommodate both equity and efficiency by solving a sequence of optimization problems, each of which maximizes a social welfare function (SWF) subject to constraints from the original model.  The SWFs are formulated using mixed integer constraints.  They successively give priority to the worst-off, the second worst-off, and so on, with the degree of priority gradually decreasing relative to utilities in the utilitarian region.  

As proof of concept, we apply our method to health resource allocation and disaster preparation problems.  The solution time is at most a matter of seconds for a given value of $\Delta$.  We find that the solutions are not only sensitive to equity considerations but reveal complex and subtle trade-offs.  This suggests that the modeling approach developed here can potentially serve as a useful mathematical tool for balancing fairness and efficiency in real-world situations.

	
\bibliographystyle{abbrvnat}
\bibliography{paperref}

\appendix

\section*{Appendix 1} 

In this Appendix, we prove that $F_k(\vu)$ satisfies the Chateauneuf-Moyes condition for $k\geq 2$.

\medskip
{\em Proof of Theorem~\ref{th:CM2}.}
	It is clear that a sufficiently small utility-invariant transfer satisfies the C-M condition when $k > t(\vu)$, because in this case $F_k(\vu)$ is simply utilitarian. We therefore need only consider the six cases illustrated in Fig.~\ref{fig:CM2}, in which $k\leq t(\vu)$.  It is convenient to write $F_k(\vu)$ in the following form:
	\[
	F_k(\vu) = t(\vu)u_{\langle 1\rangle}
	+ \sum_{i=2}^k (n-i+1)u_{\langle i\rangle} 
	+ \hspace{-1.5ex} \sum_{i=t(\vu)+1}^n \hspace{-2ex} (u_{\langle i\rangle} - \Delta)
	\]
	The resulting gain by individuals $1,\ldots\ell$, and loss by individuals $h,\ldots,n$, are indicated in Table~\ref{ta:CM2}.  In cases (b)--(f), it is clear on inspection of Fig.~\ref{fig:CM2} that the gain is more than $\epsilon$ in each case, and the loss never more than $\epsilon$. 
	In case (a), we note first that the gain can be written
	\[
	n - \frac{\ell-1}{2} - \frac{n-t(\vu)}{\ell}
	\]
	To show that the loss is no greater than the gain, it suffices to show this when $h=\ell + 1$, since $h\geq \ell+1$ and the loss is nonincreasing with respect to $h$.  Thus it suffices to show
	\[
	n - \frac{\ell-1}{2} - \frac{n-t(\vu)}{\ell} \geq
	\frac{1}{n-\ell} 
	\Big( \sum_{i=\ell +1}^k \hspace{-0.5ex} (n-i+1) + n - t(\vu) \Big)
	\]
	Since $k\leq t(\vu)$ and each term of the summation is at most $n-\ell$, it suffices to show
	\[
	n - \frac{\ell-1}{2} - \frac{n-t(\vu)}{\ell} \geq
	\frac{\big(t(\vu)-\ell\big)(n-\ell) + n - t(\vu)}{n-\ell}
	\]
	Rearranging, we obtain
	\begin{equation}
	\big( n-t(\vu)\big) \Big(\frac{1}{\ell} + \frac{1}{n-\ell} - 1 \Big) 
	\leq \frac{\ell + 1}{2}
	\label{eq:proof41}
	\end{equation}
	This inequality is clearly satisfied when the following is false:
	\begin{equation}
	\frac{1}{\ell} + \frac{1}{n-\ell} \geq 1
	\label{eq:proof40}
	\end{equation}
	We therefore assume (\ref{eq:proof40}) is true.  Since (\ref{eq:proof41}) is clearly satisfied when $\ell = 1$,  we suppose $\ell\geq 2$, in which case (\ref{eq:proof40}) implies $n < \ell^2/(\ell-1)$.  Since $\ell<h\leq n$, we can state
	\[
	\ell + 1 \leq n < \frac{\ell^2}{\ell-1}
	\]
	or $\ell^2-1 \leq n(\ell-1) < \ell^2$.  Since $n$ and $\ell$ are positive integers, this implies $n=\ell + 1$, in which case (\ref{eq:proof41}) reduces to
	\[
	\frac{\ell+1-t(\vu)}{\ell} \leq \frac{\ell+1}{2}
	\]
	This holds because $t(\vu)\geq \ell+1$, and the theorem follows.
	$\Box$
	\medskip

\begin{figure}[!h]
	\centering
	\includegraphics[trim=100 590 140 60, clip]{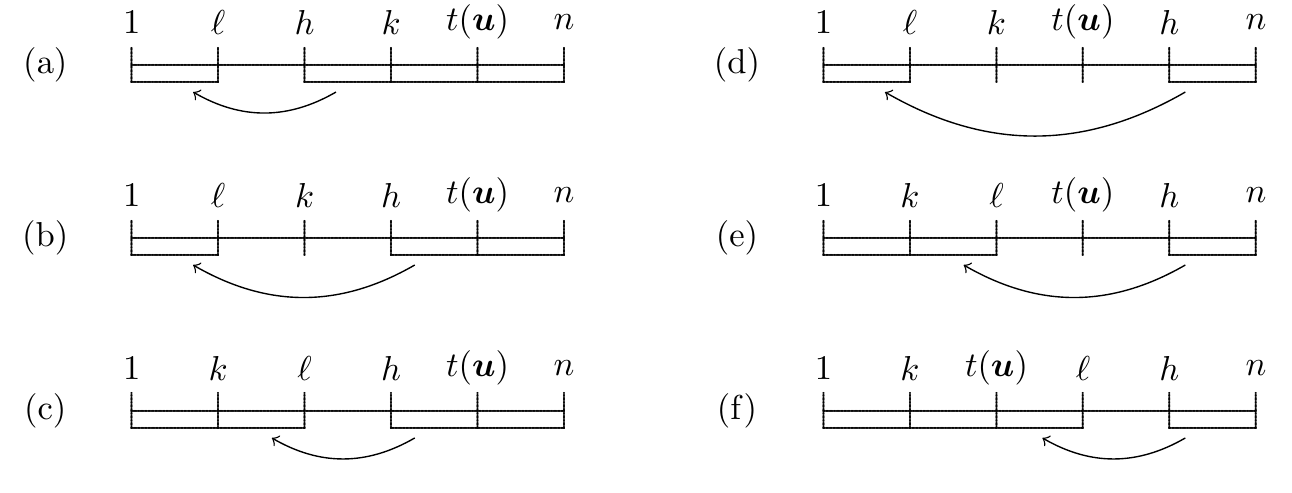}
	\caption{Illustration of proof of Theorem \ref{th:CM2}.}
	\label{fig:CM2}
\end{figure}

\begin{table}
	\centering
	\caption{Verifying the Chateauneuf-Moyes condition for $F_k(\vu)$} \label{ta:CM2}
	\vs\vs
		\begin{tabular}{cc@{\hspace{1ex}}c}
			Case & Gain & Loss \\  
			\hline \\ [-2ex]
			(a) & 
			${\ds  \frac{1}{\ell}\Big(t(\vu) + \sum_{i=2}^{\ell}(n-i+1)\Big)\epsilon } $ &
			${\ds \frac{1}{n-h+1}\Big(\sum_{i=h}^k(n-i+1) + n-t(\vu) \Big)\epsilon }$
			\vs\vs \\
			(b) & 
			${\ds \frac{1}{\ell}\Big(t(\vu) + \sum_{i=2}^{\ell}(n-i+1)\Big)\epsilon \geq \frac{t(\vu)}{\ell}\epsilon > \epsilon} $ 
			& ${\ds \frac{n-t(\vu)}{n-h+1}\epsilon < \epsilon}$ \vs\vs \\
			(c) 
			& ${\ds \frac{1}{\ell}\Big(t(\vu) + \sum_{i=2}^k(n-i+1)\Big)\epsilon \geq \frac{t(\vu)}{\ell}\epsilon > \epsilon }$ 
			& ${\ds \frac{n-t(\vu)}{n-h+1}\epsilon < \epsilon}$ \vs\vs \\
			(d) & 
			${\ds \frac{1}{\ell}\Big(t(\vu) + \sum_{i=2}^{\ell}(n-i+1)\Big)\epsilon \geq \frac{t(\vu)}{\ell}\epsilon > \epsilon} $ 
			& ${\ds \frac{n-h+1}{n-h+1}\epsilon = \epsilon}$ \vs\vs \\
			(e) & ${\ds \frac{1}{\ell}\Big(t(\vu) + \sum_{i=2}^k(n-i+1)\Big)\epsilon \geq \frac{t(\vu)}{\ell}\epsilon > \epsilon }$ 
			& ${\ds \frac{n-h+1}{n-h+1}\epsilon = \epsilon}$ \vs\vs \\ 
			(f) & ${\ds \frac{1}{\ell}\Big(t(\vu) + \sum_{i=2}^k (n-i+1) + \ell-t(\vu)\Big)\epsilon \geq \epsilon }$  
			& ${\ds \frac{n-h+1}{n-h+1}\epsilon = \epsilon}$ \vs\vs \\
			\hline
		\end{tabular}
\end{table}

\section*{Appendix 2}

In this Appendix, we derive the group-related SWFs $G_2(\vu),\ldots,G_m(\vu)$ and prove the relevant theorems.  We obtain the SWFs by treating the group members as individuals and applying the the SWFs $F_k(\vu)$ for individuals, with the assumption that all individuals in a group have the same utility.  

We begin by deriving $G_1(\vu)$.  Let $u'_{i'}$ be the utility of {\em individual} $i'$, and let $u_i$ be the utility of each individual in {\em group} $i$.  There are $n'$ individuals and $n$ groups.  Let $s_i$ be the size of group $i$, so that
\begin{equation}
n'=\sum_{i=1}^n s_i
\label{00}
\end{equation}
Then 
\[
F_1(\vu') = n' u'_{\langle 1\rangle} + (n'-1)\Delta
+ \hspace{-0.5ex} \sum_{i'=1}^{n'} (u'_{i'} - u'_{\langle 1\rangle} - \Delta)^+ 
\]
Since $u_{\langle 1\rangle}=u'_{\langle 1\rangle}$ and group $i$ has size $s_i$, we have
\[
G_1(\vu) = \Big(\sum_{i=1}^n s_i\Big)u_{\langle 1\rangle} + \Big( \sum_{i=1}^n s_i - 1\Big) \Delta + 
\sum_{i=1}^n s_i (u_i - u_{\langle 1\rangle} - \Delta)^+
\]
This is the formula used in \citeauthor{HooWil12} (\citeyear{HooWil12}).  

Hooker and Williams prove that (\ref{eq:MILPmodel1}) is a sharp representation of $P_1^{'}$, and  (\ref{eq:MILP1group-new}) a sharp representation of $P_1^{'}$ reformulated for groups.  We present a simpler proof of both theorems.  It is necessary only to prove the latter, because the former is a special case of it.

\medskip
{\em Proof of Theorems~\ref{th:sharp1} and~\ref{th:sharp1-Group}}.  
	We prove Theorem~\ref{th:sharp1-Group}, of which Theorem~\ref{th:sharp1} is a special case in which $s_i=1$ for each $i$.  It suffices to show that any inequality $z_1 \leq \va^T \vu + b$ that is valid for $P'_1$ is a surrogate (nonnegative linear combination) of inequalities in (\ref{eq:MILP1group-new}).  Let
	\[
	N = \sum_{i=1}^n s_i
	\]
	We first show that the following is a surrogate of (\ref{eq:MILP1group-new}) for any $i$:
	\begin{equation}
	z_1 \leq (N-1)\Delta  + \Big(s_i + (N-s_i)\frac{\Delta}{M}\Big)u_i  + \Big(1-\frac{\Delta}{M}\Big)\sum_{j\neq i} s_j u_j.
	\label{eq:proof10-g1sharp}
	\end{equation}
	We then show that $z_1 \geq \va^T \vu + b$ is a surrogate of the inequalities (\ref{eq:proof10-g1sharp}).  The theorem follows.
	
	To show that (\ref{eq:proof10-g1sharp}) is a surrogate of (\ref{eq:MILP1group-new}), we first note that the following is a linear combination of the upper bounds on $v_i$ in (\ref{eq:MILP1group-new}b) and (\ref{eq:MILP1group-new}c), using multipliers $1/\Delta$ and $1/(M-\Delta)$, respectively:
	\begin{equation}
	v_j \leq \frac{\Delta}{M}w + \Big(1 - \frac{\Delta}{M}\Big)u_j.
	\label{eq:proof11-g1sharp}
	\end{equation}
	We also have the following from (\ref{eq:MILP1group-new}b) and (\ref{eq:MILP1group-new}c):
	\begin{equation}
	v_i \leq u_i. \label{eq:proof12-1-g1sharp}
	\end{equation}
	\begin{equation}
	w \leq v_i. \label{eq:proof12-g1sharp}
	\end{equation}
	We now obtain the following, for any given $i$ and $j$, as a linear combination of (\ref{eq:proof11-g1sharp}) and (\ref{eq:proof12-g1sharp}), using multipliers 1 and $\Delta/M$, respectively:
	\begin{equation}
	v_j \leq \frac{\Delta}{M}v_i + \Big(1 - \frac{\Delta}{M}\Big)u_j.
	\label{eq:proof13-g1sharp}
	\end{equation}
	Finally, we obtain (\ref{eq:proof10-g1sharp}) for any given $i$ by summing (\ref{eq:MILP1group-new}a) with multiplier 1, (\ref{eq:proof12-1-g1sharp}) with multiplier 
	\[
	s_i+(N-s_i)\frac{\Delta}{M}
	\]
	and (\ref{eq:proof13-g1sharp}) over all $j\neq i$ with multiplier $s_j$.
	
	It remains to show that $z_1 \leq \bm{a}^T\vu+b$ is a surrogate of (\ref{eq:proof10-g1sharp}) for $i=1,\ldots, n$.  We first observe that $(\vu,z)=(\vec{0},(N-1)\Delta)$ is feasible in $P^{'}_1$ and must therefore satisfy $z_1 \leq \bm{a}^T \vu+b$, which implies $b\geq (N-1)\Delta$.  We can assume without loss of generality that $b=(N-1)\Delta$, since otherwise we can add an appropriate multiple of the valid inequality $0\leq b$ to obtain the desired inequality $z_1 \leq \bm{a}^T \vu+b$.  We also note that $(\vu,z)=(M,\ldots,M,NM+(N-1)\Delta)$ is feasible and must satisfy $z_1\leq \va^T \vu+(N-1)\Delta$, which means 
	\begin{equation}
	\sum_{j=1}^m a_j \geq N
	    \label{eq:proof19-g1sharp}
	\end{equation}
	Finally, we note that $(\vu,z)=(M\bm{e}_i,(N-1)\Delta + s_i(M - \Delta))$ is feasible for $P^{'}_1$, where $\bm{e}_i$ is the $i$th unit vector.  Substituting this into $z_1\leq \bm{a}^T \vu+(N-1)\Delta$, we obtain
	\begin{equation}
	a_i \geq \Big(1 - \frac{\Delta}{M}\Big)s_i
	\label{eq:proof14-g1sharp}
	\end{equation}
	Due to (\ref{eq:proof19-g1sharp}), we can suppose without loss of generality that $\sum_{j=1}^m a_j = N$, since otherwise we can add appropriate multiples of the valid inequalities $0\leq a_j$ to obtain $z_1 \leq \bm{a}^T \vu+b$.  
	
	To obtain $z\leq \bm{a}^T \vu+b$ as a surrogate of (\ref{eq:proof10-g1sharp}), we sum (\ref{eq:proof10-g1sharp}) over all $j$ using the multipliers
	\begin{equation}
	\alpha_i = \frac{M}{N\Delta} \Big( a_i - \Big(1 - \frac{\Delta}{M}\Big)s_i \Big)
	\label{eq:proof15-g1sharp}
	\end{equation}
	for each $i$.  It is easily checked that $\sum_{i=1}^m \alpha_i = 1$, so that the linear combination has the form 
	\begin{equation}
	z\leq \bm{d}^T \vu + (N-1)\Delta
	\label{eq:proof90}
	\end{equation}
	We wish to show that $\bm{d}=\bm{a}$. Note that
	\[
	d_i = \Big(s_i + N\frac{\Delta}{M} \Big) \alpha_i + \Big( 1 - \frac{\Delta}{M}\Big) s_i\sum_{j\neq i} \alpha_j
	\]
	Using the fact that $\sum_{j=1}^m \alpha_j = 1$, this becomes
	\[
	d_i = N\frac{\Delta}{M}\alpha_i + \Big(1-\frac{\Delta}{M}\Big) s_i
	\]
	which immediately reduces to $d_i=a_i$.  We conclude that (\ref{eq:proof90}) is a linear combination of the inequalities (\ref{eq:proof10-g1sharp}) using multipliers $\alpha_i$.  It remains to show that each $\alpha_i$ is nonnegative, but this follows from (\ref{eq:proof14-g1sharp}) and (\ref{eq:proof15-g1sharp}).  $\Box$
\medskip

We now derive $\bar{G}_k(\vu)$ for $k\geq 2$.  Recall that the SWF for individuals is
\begin{equation}
\bar{F}_{k'}(\vu') = (n'-k'+1)\min\{u'_{\langle 1\rangle}+\Delta, u'_{\langle k'\rangle}\}
+ \sum_{i'=k'}^{n'} \hspace{-0.5ex} (u'_{\langle i'\rangle} - \bar{u}'_{\langle 1\rangle} - \Delta)^+ 
\label{1}
\end{equation}
To obtain $\bar{G}_k(\vu)$, we again assume the individuals in each group $i$ have the same utility $u_i$.   The first individual in (\ref{1}) that belongs to group $k$ is individual
\begin{equation}
k' = 1 + \sum_{j=1}^{k-1} s_{i_j}
\label{0}
\end{equation}
Due to (\ref{00}) and (\ref{0}), the first term on the RHS of (\ref{1}) is 
\[
\Big( n' - 1 - \sum_{j=1}^{k-1}s_{i_j} + 1 \Big)
\min\{u_{\langle k\rangle} + \Delta, u_{\langle k\rangle}\}
= \Big( \sum_{i=k}^n s_{\langle i\rangle} \Big) 
\min\{u_{\langle k\rangle} + \Delta, u_{\langle k\rangle}\}
\]
since all the utilities in a group are the same.  
Thus we have
\begin{equation}
\bar{G}_k(\vu) =
\Big( \sum_{i=k}^n s_{\langle i\rangle} \Big)
\min\{u_{\langle 1\rangle} + \Delta, u_{\langle k\rangle}\}
+ \sum_{i=k}^n s_{\langle i\rangle}(u_{\langle i\rangle} - u_{\langle 1\rangle} - \Delta)^+
\label{eq:groupk-2}
\end{equation}
We show as follows that $\bar{G}_k(\vu)$ is continuous in $u_{\langle k \rangle}, \ldots, u_{\langle n \rangle}$.
\medskip

{\em Proof of Theorem~\ref{th:continuity}.}
    It suffices to show each term of (\ref{eq:groupk-2}) is a continuous function of $u_{\langle k \rangle}, \ldots, u_{\langle n \rangle}$, with $u_{\langle 1 \rangle}, \ldots, u_{\langle k-1 \rangle}$ and the corresponding group sizes $s_{\langle 1 \rangle}, \ldots, s_{\langle k-1 \rangle}$ fixed.  The first term is continuous because it is equal to a constant time the maximum of order statistics $u_{\langle 1\rangle}$ and $u_{\langle k\rangle}$, which are continuous functions of $\vu$.  Similarly, each term of the summation is a constant times the maximum of a continuous expression and zero.
$\Box$
\medskip

We can now establish that the MILP model (\ref{eq:MILPkgroup-new}) is correct.

\medskip
{\em Proof of Theorem~\ref{th:MILPGroup}.}
We first show that given any $(\vu,z_k)$ that is feasible for (\ref{eq:seq2}), where $u_{i_j}=\bar{u}_{i_j}$ for $j=1,\ldots, k-1$, there exist $\bm{v}, \bm{\delta},\bm{\epsilon},w,\sigma$ for which $(\vu,z_k,\bm{v},\bm{\delta},\bm{\epsilon},w,\sigma)$ is feasible for (\ref{eq:MILPkgroup-new}).  Constraint (\ref{eq:MILPkgroup-new}$j$) follows directly from (\ref{eq:seq2}$c$).  To satisfy the remaining constraints in (\ref{eq:MILPkgroup-new}), we assign values to $\bm{v}, \bm{\delta},\bm{\epsilon},w,\sigma$ as in (\ref{eq:assign}),
where $\kappa$ is an arbitrarily chosen index in $I_k$ such that $u_{\kappa}=u_{\langle k\rangle}$.
It is easily checked that these assignments satisfy constraints (\ref{eq:MILPkgroup-new}$b$)--(\ref{eq:MILPkgroup-new}$h$).  They satisfy (\ref{eq:MILPkgroup-new}$i$) because (\ref{eq:seq2}$b$) implies that $u_{\kappa}\geq \bar{u}_{i_{k-1}}$.  To show they satisfy (\ref{eq:MILPkgroup-new}$a$), we note that (\ref{eq:MILPkgroup-new}$a$) is implied by (\ref{eq:seq2}$a$) because  $\min\{\bar{u}_{i_1}+\Delta,u_{\kappa}\}\leq \sigma$ and $(u_i-\bar{u}_{i_1}-\Delta)^+\leq v_i$ for $i\in I_k$.  Since (\ref{eq:seq2}$a$) is satisfied by $(\vu,z)$, it follows that (\ref{eq:MILPkgroup-new}$a$) is satisfied by (\ref{eq:assign}).

For the converse, we show that for any $(\vu,z_k,\bm{v},\bm{\delta},\bm{\epsilon},w,\sigma)$ that satisfies (\ref{eq:MILPkgroup-new}), $(\vu,z_k)$ satisfies (\ref{eq:seq2}). Constraint (\ref{eq:seq2}$b$) follows from (\ref{eq:MILPmodel}$f$) and (\ref{eq:MILPmodel}$i$), and (\ref{eq:seq2}$c$) is identical to (\ref{eq:MILPkgroup-new}$j$).  To verify that (\ref{eq:seq2}$a$) is satisfied, we let $\kappa$ be the index for which $\epsilon_{\kappa}=1$, which is unique due to (\ref{eq:MILPkgroup-new}$g$).  It suffices to show that (\ref{eq:MILPkgroup-new}$a$) implies (\ref{eq:seq2}$a$) when the remaining constraints of (\ref{eq:MILPkgroup-new}) are satisfied.  For this it suffices to show that 
\begin{equation}
\sigma \leq \min\{\bar{u}_{i_1}+\Delta, u_{\kappa}\}
\label{eq:proof1group}
\end{equation}
\begin{equation}
v_i \leq (u_i - \bar{u}_{i_1} - \Delta)^+, \; i\in I_k
\label{eq:proof2group}
\end{equation}
(\ref{eq:proof1group}) follows from ($d$), ($e$), and ($f$) of (\ref{eq:MILPkgroup-new}).  (\ref{eq:proof2group}) follows from ($b$) and ($c$) of (\ref{eq:MILPkgroup-new}).  This proves the theorem.
$\Box$
\medskip

Finally, we show that (\ref{eq:valid0-group})--(\ref{eq:valid-group}) are valid inequalities for the group version of $P'_k$ for $k\geq 2$.

\medskip
{\em Proof of Theorem~\ref{th:valid-group}.}  
It suffices to show that for any $(\vu,z_k,\bm{v},\bm{\delta},\bm{\epsilon},w)$ that satisfies (\ref{eq:MILPkgroup-new}), where $u_{i_j}=\bar{u}_{i_j}$ for \mbox{$j=1,\ldots,k-1$}, the vector $\vu$ satisfies (\ref{eq:valid0}) and (\ref{eq:valid}).  Since we know from Theorem~\ref{th:MILPGroup} that $\vu$ is feasible in (\ref{eq:seq2-group}), it suffices to show that (\ref{eq:seq2-group}) implies (\ref{eq:valid0-group}) and (\ref{eq:valid-group}).  To derive (\ref{eq:valid0-group}), we write (\ref{eq:seq2-group}$a$) as
\begin{equation}
z_k \leq \sum_{i\in I_k} s_i\Big( \min\{\bar{u}_{i_1}+\Delta, u_{\langle k\rangle}\}
+(u_i - \bar{u}_{i_1} - \Delta)^+ \Big)
\end{equation}
For any term $i$ in the summation, we consider two cases.  If $u_i\leq \bar{u}_{i_1}+\Delta$, then $u_{\langle k\rangle} \leq \bar{u}_{i_1}+\Delta$ (because $u_{\langle k\rangle}\leq u_i$), and the term reduces to $s_iu_{\langle k\rangle}$.  If $u_i> \bar{u}_{i_1}+\Delta$, term $i$ becomes
\[
s_i\Big(\min\{\bar{u}_{i_1}+\Delta, u_{\langle k\rangle}\}
+(u_i - \bar{u}_{i_1} - \Delta)\Big)
= s_i\Big(\min\{0, u_{\langle k\rangle} - \bar{u}_{i_1} - \Delta\} + u_i\Big) \leq s_iu_i
\]
In either case, term $i$ is less than or equal to $u_i$, and (\ref{eq:valid0-group}) follows.

To establish (\ref{eq:valid-group}), it is enough to show that (\ref{eq:valid-group}) is implied by (\ref{eq:seq2-group}) for each $i\in I_k$.  We consider the same two cases as before.

Case 1: $u_i-\bar{u}_{i_1}\leq\Delta$, which implies $u_{\langle k\rangle}-\bar{u}_{i_1}\leq\Delta$.  Since $\vu$ satisfies (\ref{eq:seq2-group}$a$), we have
\begin{equation}
z_k \leq \Big(\sum_{j\in I_k} s_j\Big)u_{\langle k\rangle} 
+ \hspace{-2ex} \sum_{\substack{j\in I_k\setminus\{i\}\\u_j-\bar{u}_{i_1}>\Delta}} \hspace{-2.5ex} s_j(u_j-\bar{u}_{i_1}-\Delta)
\label{eq:proof040-group}
\end{equation}
It suffices to show that this implies 
\begin{equation}
z_k \leq \Big(\sum_{j\in I_k} s_j\Big)u_i +
\beta \Big( \hspace{-2ex} \sum_{\substack{j\in I_k\setminus\{i\}\\u_j-\bar{u}_{i_1}\leq\Delta}} \hspace{-2.5ex} s_j(u_j-\bar{u}_{i_{k-1}})
\hspace{0.5ex} + \hspace{-2ex} \sum_{\substack{j\in I_k\setminus\{i\}\\u_j-\bar{u}_{i_1}>\Delta}} \hspace{-2.5ex} s_j(u_j-\bar{u}_{i_{k-1}}) \Big),
\label{eq:proof041-group}
\end{equation}
because (\ref{eq:proof041-group}) is equivalent to the desired inequality (\ref{eq:valid-group}).  But (\ref{eq:proof040-group}) implies (\ref{eq:proof041-group}) because $u_{\langle k\rangle} \leq u_i$ by definition of $u_{\langle k\rangle}$, $u_j-\bar{u}_{i_{k-1}}\geq 0$ for all $j\in I_k$ due to (\ref{eq:seq2-group}$b$), and (\ref{eq:proof40}) for all $j\in I_k$. 
 
Case 2: $u_i-\bar{u}_{i_1}> \Delta$.  It again suffices to show that (\ref{eq:seq2-group}) implies (\ref{eq:proof041-group}).  Due to the case hypothesis, we have from (\ref{eq:seq2-group}$a$) that 
\[
z_k \leq \Big(\sum_{j\in I_k} s_j\Big)\min\{\bar{u}_1+\Delta,u_{\langle k\rangle}\}
+ s_i(u_i-\bar{u}_{i_1}-\Delta) 
+ \hspace{-2.5ex} \sum_{\substack{j\in I_k\setminus\{i\}\\u_j-\bar{u}_{i_1}>\Delta}} \hspace{-2.5ex} s_j(u_j-\bar{u}_{i_1}-\Delta)
\]
This can be written
\[
z_k \leq \Big(\sum_{j\in I_k} s_j\Big)u_i 
- \Big(\sum_{j\in I_k} s_j\Big)\Big( u_i - \min\{\bar{u}_1+\Delta,u_{\langle k\rangle}\} \Big)
+ s_i(u_i-\bar{u}_{i_1}-\Delta) 
+ \hspace{-2.5ex} \sum_{\substack{j\in I_k\setminus\{i\}\\u_j-\bar{u}_{i_1}>\Delta}} \hspace{-2.5ex} s_j(u_j-\bar{u}_{i_1}-\Delta)
\]
which can be written
\begin{equation}
z_k \leq \Big(\sum_{j\in I_k} s_j\Big)u_i 
- \Big( \hspace{-2ex} \sum_{j\in I_k\setminus\{i\}} \ \hspace{-3ex} s_j\Big)\Big( u_i - \min\{\bar{u}_1+\Delta,u_{\langle k\rangle}\} \Big)
- s_i\Big( \bar{u}_1 + \Delta - \min\{\bar{u}_1+\Delta,u_{\langle k\rangle} \} \Big)
+ \hspace{-2.5ex} \sum_{\substack{j\in I_k\setminus\{i\}\\u_j-\bar{u}_{i_1}>\Delta}} \hspace{-2.5ex} s_j(u_j-\bar{u}_{i_1}-\Delta)
\label{eq:proof45-group}
\end{equation}
The second term is nonpositive because $u_i>\bar{u}_1+\Delta$ by the case hypothesis, and $u_i\geq u_{\langle k\rangle}$.  The third term is clearly nonpositive.  Thus (\ref{eq:proof45-group}) implies (\ref{eq:proof041-group}) because $u_j-\bar{u}_{i_{k-1}}\geq 0$ and (\ref{eq:proof401}) holds for $j\in I_k$ as before. $\Box$

\end{document}